\documentclass[12pt]{article}
\usepackage{amsmath,amsthm,amssymb,caption,pgfplots}
\usepackage{wrapfig,float}
\usepackage{todonotes}
\usepackage{algorithm,algorithmicx,algpseudocode}
\usepackage{framed}
\usepackage{hyperref}
\usepackage{authblk}
\usepackage{fullpage}
\usepackage{color}
\title{A Quadratically Convergent Algorithm for\\Structured Low-Rank Approximation}

\author[1]{\'Eric Schost}
\author[1,2,3]{Pierre-Jean Spaenlehauer}
\affil[1]{ORCCA and CS Department, Western University, London, ON Canada}
\affil[2]{Max Planck Institute for Mathematics, Bonn, Germany}
\affil[3]{Inria, CNRS, Universit\'e de Lorraine, France}

\date{}

\newcommand{\R}{\mathbb R}

\newcommand{\N}{\mathbb N}
\newcommand{\E}{\mathbb E}

\newcommand{\Sphere}{\mathbb S}
\newcommand{\norm}[1]{\left\lVert #1 \right\rVert}
\DeclareMathOperator{\GCD}{GCD}

\DeclareMathOperator{\SVD}{SVD}

\DeclareMathOperator{\trace}{trace}
\DeclareMathOperator{\rank}{rank}
\DeclareMathOperator{\SubRes}{Syl}

\DeclareMathOperator{\Ker}{Ker}
\DeclareMathOperator{\Hom}{Hom}
\DeclareMathOperator{\dist}{dist}

\DeclareMathOperator{\matsp}{\mathcal M}
\DeclareMathOperator{\argmin}{argmin}
\DeclareMathOperator{\im}{Im}
\DeclareMathOperator{\codim}{codim}
\newcommand{\detvar}{\mathcal D}
\newcommand{\mani}{\mathcal V}
\newcommand{\maniinter}{\mathcal W}

\newtheorem{theo}{Theorem}
\numberwithin{theo}{section}
\newtheorem*{theobis}{Theorem}
\newtheorem{lem}[theo]{Lemma}

\newtheorem{coro}[theo]{Corollary}
\newtheorem{prop}[theo]{Proposition}
\newtheorem{defi}[theo]{Definition}

\newcounter{condalpha}
\newcounter{countproblem}

\newcommand*\hexbrace[2]{%
  \underset{#2}{\underbrace{\rule{#1}{0pt}}}}

\newenvironment{problem}[1]
{
  \refstepcounter{countproblem}
  \begin{center}
    \begin{minipage}{0.9\linewidth}
      \begin{framed}
        {\bf Problem \arabic{countproblem} - {#1}.}
}
{  
      \end{framed}
    \end{minipage}
  \end{center}
}

\begin{document}
\maketitle

\begin{abstract} Structured Low-Rank Approximation is a problem arising
	in a wide range of applications in Numerical Analysis and
        Engineering Sciences. Given an input matrix $M$, the goal is
        to compute a matrix $M'$ of given rank $r$ in a linear or
        affine subspace $E$ of matrices (usually encoding a specific
        structure) such that the Frobenius distance $\norm{M-M'}$ is
        small. We propose a Newton-like iteration for solving this
        problem, whose main feature is that it converges locally
        quadratically to such a matrix under mild transversality
        assumptions between the manifold of matrices of rank $r$ and
        the linear/affine subspace $E$. We also show that the distance
        between the limit of the iteration and the optimal solution of
        the problem is quadratic in the distance between the input
        matrix and the manifold of rank $r$ matrices in $E$. To
        illustrate the applicability of this algorithm, we propose a
        {\tt Maple} implementation and give experimental results for several
        applicative problems that can be modeled by Structured
        Low-Rank Approximation: univariate approximate GCDs (Sylvester
        matrices), low-rank Matrix completion (coordinate spaces) and
        denoising procedures (Hankel matrices). Experimental
        results give evidence that this all-purpose algorithm is
        competitive with state-of-the-art numerical methods dedicated
        to these problems.
\end{abstract}



%%%%%%%%%%%%%%%%%%%%%%%%%%%%%%%%%%%%%%%%%%%%%%%%%%%%%%%%%%%%
%%%%%%%%%%%%%%%%%%%%%%%%%%%%%%%%%%%%%%%%%%%%%%%%%%%%%%%%%%%%
%%%%%%%%%%%%%%%%%%%%%%%%%%%%%%%%%%%%%%%%%%%%%%%%%%%%%%%%%%%%

\section{Introduction}

\subsection{Motivation and problem statement}

In a wide range of applications (data fitting, symbolic-numeric
computations, signal processing, system and control theory,\ldots),
the problem arises of computing low rank approximations of matrices
under linear constraints; this central question is known as
\emph{Structured Low-Rank Approximation} (abbreviated SLRA). Quoting
Markovsky~\cite{Mar08}: \emph{behind every linear data modeling
  problem there is a (hidden) low-rank approximation problem: the
  model imposes relations on the data which render a matrix
  constructed from exact data rank deficient.}  We refer the reader to
\cite{Mar08} for an overview of the vast extent of fields where SLRA
arises in a natural way.

Let $\matsp_{p,q}(\R)$ denote the space of $p\times q$ matrices with real
entries, endowed with the inner product
$$\langle M_1,M_2\rangle=\trace(M_1\cdot M_2^\intercal);$$ this vector
space inherits the Frobenius norm $\norm{M}=\sqrt{\langle M,M\rangle}$
deduced from this inner product. For $r \in \N$, let further
$\detvar_r\subset\matsp_{p,q}(\R)$ denote the set of matrices of size
$p\times q$ and of rank equal to $r$; this is both a semi-algebraic
set and an analytic manifold in $\matsp_{p,q}(\R)$ of codimension
$(p-r)(q-r)$~\cite[Proposition 1.1]{BruVet88}. The Structured
Low-Rank Approximation Problem can be stated as follows:

\begin{problem}{Structured Low-Rank Approximation (SLRA)}\label{prob:SLRA}
  Let $E\subset \matsp_{p,q}(\R)$ be an affine subspace of
  $\matsp_{p,q}(\R)$, let $M\in E$ be a matrix and let $r\in\N$ be a
  integer. Find a matrix $M^\star\in E\cap \detvar_r$ such that
  $\norm{M-M^\star}\text{ is ``small''}.$
\end{problem}

The problem is not entirely specified yet, since we have to state what
``small'' means.  Actually, several variants of this problem can be
found in the litterature (for instance this problem can be stated
similarly for other norms). One way to approach SLRA is as an
optimization problem, by looking for the matrix $M^\star$ in $E \cap
\detvar_r$ which minimizes $\norm{M-M^\star}$, \emph{i.e.} such that
$\norm{M-M^\star}=\dist(M,E \cap \detvar_r)$, where $\dist(M,S)$
denotes the distance of $M$ to the set $S$. Let us denote by $\Pi_{E
  \cap \detvar_r}$ the orthogonal projection $\Pi_{E \cap
  \detvar_r}(M)=\argmin_{M^\star\in E \cap
  \detvar}(\norm{M-M^\star}),$ which is well-defined and continuous in
a neighborhood of $E\cap \detvar_r$. Then, the optimization form of
the SLRA problem precisely amounts to computing $\Pi_{E \cap
  \detvar_r}(M)$.

On another hand, it may also be sufficient to compute a matrix
$M'$ whose distance to the optimal solution is small with respect
to $\dist(M,E \cap \detvar_r)$. This is a mild relaxation of the
optimization form of the problem, and it seems to be sufficient for
many applications. Indeed, the SLRA problem often arises in situations
where an exact structured matrix has been perturbed by some noise, and
SLRA can be viewed as a denoising procedure; in this context, the
original matrix may not be the optimal solution of the underlying SLRA
problem and therefore computing $M'\in E \cap \detvar_r$ such
that $\norm{M-M'}$ is small may be sufficient if the error is
controlled.

%%%%%%%%%%%%%%%%%%%%%%%%%%%%%%%%%%%%%%%%%%%%%%%%%%%%%%%%%%%%

\subsection{Main results}

We propose an iterative algorithm, called {\tt NewtonSLRA}, solving the second form of the SLRA
problem with proven quadratic convergence, under mild transversality
conditions on $E$ and $\detvar_r$. Given an input matrix $M$ in $E$,
the output of the algorithm is a matrix $M'$ in $E \cap \detvar_r$
which is a good approximation of the optimal
$\Pi_{E\cap\detvar_r}(M)$, in the sense that the distance
$\norm{\Pi_{E\cap\detvar_r}(M)-M'}$ is quadratic in $\dist(M,E \cap
\detvar_r)=\norm{\Pi_{E\cap\detvar_r}(M)-M}$.

An iteration of the algorithm relies mainly on a Singular Value
Decomposition, plus a few further linear algebra operations. It is not
our goal in this paper to analyze the numerical accuracy of our
algorithm in floating-point arithmetic. For this reason, we would like
to state the complexity analysis in terms of arithmetic operations
$+,-,\times,\div$ on real numbers. We can achieve this for all
operations except the Singular Value Decomposition, which is an
iterative process in itself
(see~\cite[Ch.~45-46]{hogben2005handbook}). As a result, in our cost
analysis, we isolate the cost induced by the Singular Value
Decomposition, and count all other arithmetic operations at unit cost.

In all that follows, if $M$ is a matrix in $\matsp_{p,q}(\R)$, we
let  $B_\rho(M)\subset \matsp_{p,q}(\R)$ denote the open ball centered
at $M$ and of radius $\rho$.

\begin{theobis}
  The algorithm {\tt NewtonSLRA} computes a function
  $\varphi$ defined on an open neighborhood $U\supset\detvar_r$, and with codomain $E$,
  verifying the following property:

  Let $\zeta$ be in $E \cap \detvar_r$ such that $E$ and $\detvar_r$
  intersect transversally at $\zeta$. There exist
  $\nu,\gamma,\gamma'>0$ such that, for all $M_0$ in $E \cap
  B_\nu(\zeta)$, the sequence $(M_i)$ given by $M_{i+1}=\varphi(M_i)$
  is well-defined and converges towards a matrix $M_\infty\in E \cap
  \detvar_r$ and
  \begin{itemize} 
  \item $\norm{M_{i+1}-M_\infty}\leq \gamma \norm{M_{i}-M_\infty}^2$ for all $i\geq 0$;
  \item $\norm{\Pi_{E \cap \detvar_r}(M_0)-M_\infty}\leq \gamma' \dist(M_0,E \cap \detvar_r)^2$.
  \end{itemize}
  Moreover, the function $\varphi$ can be computed by means of a
  Singular Value Decomposition of the input matrix $M$, plus $O(\min(pqd(p-r)(q-r)+pqr, 
pqr(pq-d)(p+q-r)))$ arithmetic operations, with $d=\dim(E)$.
\end{theobis}
To the best of our knowledge, this is the first algorithm for SLRA
with proven local quadratic convergence. We can actually give
explicit estimates of the constants $\gamma$ and $\gamma'$, which
depend on the incidence angle between $\detvar_r$ and $E$ around
$\zeta$ and on the second derivatives of the projection operators
$\Pi_{\detvar_r}$ and $\Pi_{E \cap \detvar_r}$.

Algorithm {\tt NewtonSLRA} is a variant of a lift-and-project
technique which was introduced by Cadzow \cite{Cad88}. However,
instead of projecting orthogonally from $\detvar_r$ back to $E$, we
choose a direction of projection which is tangent to the determinantal
variety $\detvar_r$, in the spirit of Newton's method. The algorithm
relies on the Singular Value Decomposition in order to achieve the
``lifting'' step.

Let us denote by $\Phi$ the limit mapping, given by
$\Phi(M)=M_\infty$, for $M_\infty$ as in the above theorem. The
following theorem states that $\Phi$ behaves to the first order as the
operator $\Pi_{E \cap \detvar_r}$, which returns the optimal solution
of the SLRA problem. In what follows, for $\zeta$ in $E \cap
\detvar_r$, we denote by $T_\zeta (E \cap \detvar_r)^0$ the tangent
vector space to $E \cap \detvar_r$ at $\zeta$ (which is well-defined
as soon as $\detvar_r$ and $E$ intersect transversally at $\zeta$).

\begin{theobis}
 The limit operator $\Phi$ is well-defined and continuous around any
 point $\zeta\in E \cap \detvar_r$ such that $\detvar_r$ and $E$
 intersect transversally at $\zeta$, and $\Phi$ satisfies
 $\Phi(\zeta)=\Pi_{E \cap \detvar_r}(\zeta)=\zeta.$ Moreover, $\Phi$
 is differentiable at $\zeta$ and
 $$
 D\Phi(\zeta)=D\Pi_{E \cap \detvar_r}(\zeta)=\Pi_{T_\zeta (E \cap \detvar_r)^0}.
 $$
\end{theobis}

These results actually hold more generally than in the SLRA context:
the manifold $\mathcal D_r$ of rank $r$ matrices could be replaced by
any manifold $\mani$ such that the projection $\Pi_\mani$ is of class
$C^2$ and can be computed efficiently, and such that for any point
$v\in\mani$ a basis of the normal space $N_v\mani$ can be obtained.
In the context of SLRA where $\mani=\detvar_r$, the projection on
$\detvar_r$ and a description of the normal space can be obtained from
the Singular Value Decomposition.

Our algorithm \texttt{NewtonSLRA} is suitable for practical
computations: to illustrate its efficiency, we have implementated it
in {\tt Maple} and applied it in different contexts:
\begin{itemize}
\item univariate approximate GCDs;
\item low-rank matrix completion;
\item low-rank approximation of Hankel matrices.
\end{itemize}
For all of these contexts, we provide experimental results and compare
it with state-of-the-art techniques.

%%%%%%%%%%%%%%%%%%%%%%%%%%%%%%%%%%%%%%%%%%%%%%%%%%%%%%%%%%%%

\subsection{Related works}

Structured low-rank approximation and its applications have led to
huge amounts of work during the last decades, from different
perspectives. One of the first iterative methods for computing SLRA is
due to Cadzow and is based on alternating projections~\cite{Cad88}; it
has a linear rate of convergence \cite{LewMal08}.

A different approach is based on optimization techniques to
approximate the nearest low-rank matrix. The difficulty in this
setting lies in the implicit description of the problem and of the
feasible set. It has led to a large family of algorithms, see \emph{e.g.} 
\cite{ChuFunPle03} and references therein.

Several particular cases of SLRA problems have also been deeply
investigated, and specific algorithms have been proposed for these
special cases.  For instance, the \emph{matrix completion} problem
asks for unknown values of a matrix in order to satisfy a rank
condition \cite{Van13}. In particular, this computational question
appears in machine learning or in compressed sensing problems, and
convex optimization techniques have been developped in this context,
see \emph{e.g.} \cite{candes2010power,candes2010matrix,Rec11}.  Techniques of alternating minimizations
for SLRA, leading to linear (also called \emph{geometric})
convergence have been introduced in \cite{2012arXiv1212.0467J}.

Structured Low-Rank Approximation is also underlying several problems
in hybrid symbo\-lic-numerical computations. The notion of
\emph{quasi-GCD} introduced in \cite{schonhage1985quasi} shows how to
compute GCDs by using floating-point computations and has led to
developments in the last decades of different notions of approximate
GCD. In particular, degree conditions on the approximate univariate or
multivariate GCD can be expressed by a rank condition in a linear
space of matrices (see \emph{e.g}
\cite{karmarkar1996approximate,karmarkar1998approximate,KalYanZhi06,zeng2004approximate,li2005fast,winkler2008structured}).
Certified techniques \cite{EmiGalLom97} and geometric approaches
\cite{pan2001computation} (by perturbing the roots instead of
perturbing the coefficients) have also been developed.

Approximate multivariate factorization also involves a linear space of
matrices (Ruppert matrices) and can be modeled by SLRA
\cite{gao2004approximate,KalMayYanZhi08}. The relation between the
ranks of Ruppert matrices and the reducibility of multivariate
polynomials follows from a criterion introduced in
\cite{ruppert1998reducibility}.

Denoising procedures in Signal Processing often involve low rank
approximation in the linear space of Hankel matrices. Dedicated
techniques for this task have been designed and analyzed in
\cite{park1999low}.

Another line of work motivated by the matrix completion problem has been
initiated in \cite{AbsAmoMey12} by designing a Newton-like method for computing
the optima of functions defined over Riemannian manifolds. 
Other optimization techniques such as the \emph{Structured Total Least Squares} approach have also been applied to the SLRA problem and can be applied to different matrix norms \cite{rosen1998structured,park1999low,KalYanZhi07}.

In \cite{DraHorOttStuTho13}, the authors show several algebraic and
geometric properties of the critical points of the Euclidean distance
function on an algebraic variety. For instance, a connection is
exhibited between the number of complex critical points and the
degrees of the Chern classes of the variety. Algebraic methods for
solving the SLRA optimization problem from this viewpoint have been
investigated in \cite{OttSpaStu13}, with a special focus on generic
linear spaces $E$ and on SLRA problems occurring in approximate GCD
and symmetric tensor decompositions.

%%%%%%%%%%%%%%%%%%%%%%%%%%%%%%%%%%%%%%%%%%%%%%%%%%%%%%%%%%%%

\subsection{Organization of the paper}

Section \ref{sec:prelim} introduces the main tools that will be used
throughout this paper.  In Section \ref{sec:algo}, we describe the
algorithm {\tt NewtonSLRA}, we prove its correctness and derive the
complexity of each of its iteration.  The main result is the proof of
the local quadratic rate of convergence of {\tt NewtonSLRA} in Section
\ref{sec:proofquad}. Finally, we show the experimental behavior of
    {\tt NewtonSLRA} in Section \ref{sec:expe} and apply it to three
    different applicative contexts: univariate approximate GCD,
    low-rank matrix completion, and low-rank approximation of Hankel
    matrices.

%%%%%%%%%%%%%%%%%%%%%%%%%%%%%%%%%%%%%%%%%%%%%%%%%%%%%%%%%%%%

%%%%%%%%%%%%%%%%%%%%%%%%%%%%%%%%%%%%%%%%%%%%%%%%%%%%%%%%%%%%
%%%%%%%%%%%%%%%%%%%%%%%%%%%%%%%%%%%%%%%%%%%%%%%%%%%%%%%%%%%%
%%%%%%%%%%%%%%%%%%%%%%%%%%%%%%%%%%%%%%%%%%%%%%%%%%%%%%%%%%%%

\section{Preliminaries}
\label{sec:prelim}
Our algorithm combines features of the alternating projections
algorithm and of Newton's method for solving underdetermined
systems. In this section, we introduce basic ingredients used in those
previous algorithms that will be reused here, and present the basics
of alternating projections techniques and Newton iteration for
comparison purposes.

%%%%%%%%%%%%%%%%%%%%%%%%%%%%%%%%%%%%%%%%%%%%%%%%%%%%%%%%%%%%

\subsection{Notations and basic facts}

Throughout this paper, if $E$ is an affine space, $E^0$ denotes the
underlying vector space, so that $E=x+E^0$, for any $x$ in $E$. In
particular, if $\mani$ is a manifold or an algebraic set lying in a
Euclidean space, and $x$ is in $\mani$, then $T_x\mani$ denotes the
affine space that is tangent to $\mani$ at $x$ and the underlying
vector space is denoted by $T_x \mani^0$; thus $T_x \mani =
x+T_x\mani^0$.  Similarly, the normal space $N_x \mani$ to $\mani$ at
$x$ is given by $N_x \mani = x + N_x \mani^0$, where $N_x \mani^0$ is
the orthogonal complement of $T_x \mani^0$.

Recall next our definition of the projection operator $\Pi_\mani$ on
the manifold $\mani$. For a proof of the following properties,
see~\cite[Lemma 4]{LewMal08}.

\begin{lem}\label{prop:diffpi}
 Let $\E$ be a Euclidean space and let $\mani\subset\E$ be a manifold
 of class $C^k$ with $k\geq 2$. There exists an open neighborhood $U$ of
 $\mani$ such that the projection
 $$\Pi_\mani(x)=\argmin\{\norm{y-x}:y\in \mani\}$$ is well-defined on
 $U$. Moreover, $\Pi_\mani$ is of class $C^{k-1}$ around any
 point $\zeta\in \mani$ and
 $$\forall \zeta\in\mani, D\Pi_\mani(\zeta)=\Pi_{T_{\Pi_\mani(\zeta)} \mani^0}.$$
\end{lem}

We will need further results regarding the projection $\Pi_\mani$;
they will be obtained under suitable transversality assumptions. For
definiteness, let us recall the definition of transversality.
\begin{defi}
  Let $\E$ be a Euclidean space, let $\mani\subset\E$ be a manifold of
  class $C^1$, and let $E$ be an affine subspace of $\E$. We
  say that $E$ and $\mani$ intersect transversally at $\zeta\in E\cap
  \mani$ if
  $$\codim(E^0 \cap T_\zeta\mani^0)=\codim(E^0)+\codim(T_\zeta\mani^0).$$
\end{defi}
In particular, suppose that $\E$ has dimension $n$, $\mani$ has
dimension $s$, and $E$ has dimension $d$; then, a necessary condition
for them to intersect transversally is that $s+d \ge n$. In that case,
remark that $E^0 \cap T_\zeta\mani^0$ has dimension $t=s+d-n$.

Under such a transversality assumption, we obtain the following
results on the existence of smooth bases of several vector spaces.

\begin{lem}\label{lemma:bases}
  Let $\E$ be a Euclidean space of dimension $n$, let $E$ be an affine
  subspace of $\E$ of dimension $d<n$, and let $\mani\subset\E$ be a
  manifold of dimension $s$ and of class $C^k$ with $k\geq 1$. Suppose that $E$
  and $\mani$ intersect transversally at a point $\zeta\in E
  \cap\mani$; let further $t=s+d-n$ be the dimension of $E^0 \cap
  T_\zeta\mani^0$.

  Then, there exists an open neighborhood $U$ of $\zeta$ and functions
  $e_1,\dots,e_t$, $e'_{t+1},\dots,e'_d$ and $e''_{t+1},\dots,e''_s$,
  all of class $C^{k-1}: U \to \E$, such that the following
  holds:
  \begin{itemize}
  \item for $x$ in $U$, the families $(e_1(x),\dots,e_t(x))$,
    $(e_1(x),\dots,e_t(x),e'_{t+1}(x),\dots,e'_d(x))$ and \sloppy
    $(e_1(x),\dots,e_t(x),e''_{t+1}(x),\dots,e''_s(x))$ are
    orthonormal
  \end{itemize}
  and, for $x$ in $\mani \cap U$:
  \begin{itemize}
  \item the intersection $E \cap T_x \mani$ is not empty;
  \item $(e_1(x),\dots,e_t(x))$ is a basis $E^0 \cap T_x\mani^0$;
  \item $(e_1(x),\dots,e_t(x),e'_{t+1}(x),\dots,e'_d(x))$ is a basis of $E^0$;
  \item $(e_1(x),\dots,e_t(x),e''_{t+1}(x),\dots,e''_s(x))$ is a basis of $T_x\mani^0$.
  \end{itemize}
\end{lem}
\begin{proof}
  There exist linear forms $\ell_1,\dots,\ell_{n-d}$ and constants
  $b_1,\dots,b_{n-d}$ such that for all $u$ in $\E$, $u$ is in $E$ if
  and only if $\ell_i(u)=b_i$ for all $i$ in
  $\{1,\dots,n-d\}$. Similarly, taking the gradients of implicit
  equations $\varphi_1=\dots=\varphi_{n-s}=0$ that define $\mani$
  around $\zeta$, we see that there exists a neighborhood $U$ of
  $\zeta$, functions $\ell'_1,\dots,\ell'_{n-s}: U\times \E \to \R$ of
  class $C^{k-1}$ in $x \in U$ and linear in $u \in \E$, and functions
  $b'_1,\dots,b'_{n-s}: U \to \R$ of class $C^{k-1}$ such
  that for $x$ in $\mani \cap U$, $u \in \E$ belongs to $T_x\mani$ if
  and only if $\ell'_j(x,u)=b'_j(x)$ for all $j$ in $\{1,\dots,n-s\}$.

  Thus, for a given $x$ in $\mani \cap U$, $u$ belongs to $E \cap
  T_x \mani$ if and only if the linear equations $\ell_i(u)=b_i$ and
  $\ell'_j(x,u) = b'_j(x)$ are satisfied for all $i$ in
  $\{1,\dots,n-d\}$ and $j$ in $\{1,\dots,n-s\}$. Call
  $\eta_1,\dots,\eta_{2n-s-d}$ the linear forms defining the
  homogeneous part of these equations; the corresponding homogeneous
  linear system $\eta_i=0$ defines $E^0 \cap T_x \mani^0$. The
  transversality assumption shows that for $x=\zeta$, the
  $(2n-s-d)\times n$ matrix of this system has full rank $2n-s-d$. By
  continuity, this remains true for $x$ in a neighborhood of $\zeta$,
  and for such $x$, $E \cap T_x \mani$ is not empty. Up to restricting
  $U$, this proves the second item.

  Applying Cramer's formulas, we can deduce from
  $(\varphi_1,\dots,\varphi_{n-s})$ and $(\ell_1,\dots,\ell_{n-d})$
  functions $(\varepsilon_1,\dots,\varepsilon_t)$,
  $(\varepsilon'_{t+1},\dots,\varepsilon'_d)$,
  $(\varepsilon''_{t+1},\dots,\varepsilon''_s)$, with all
  $\varepsilon_i,\varepsilon'_j,\varepsilon''_k$ of class $C^{k-1}: U \to
  \E$, such that for $x$ in $U$, the vector families $(\varepsilon_1(x),\dots,\varepsilon_t(x))$,
  $(\varepsilon_1(x),\dots,\varepsilon_t(x),\varepsilon'_{t+1}(x),\dots,\varepsilon'_d(x))$
  and
  $(\varepsilon_1(x),\dots,\varepsilon_t(x),\varepsilon''_{t+1}(x),\dots,\varepsilon''_s(x))$
  are nullspace bases for respectively
  $$
  \begin{array}{@{~}c@{~}c@{~}ccl}
    \ell_1(u)=\dots=\ell_{n-d}(u)&=&D\varphi_1(x)(u)=\cdots=D\varphi_{n-s}(x)(u)&=&0\\
    &&\ell_1(u)=\dots=\ell_{n-d}(u)&=&0\\
    &&D\varphi_1(x)(u)=\cdots=D\varphi_{n-s}(x)(u)&=&0.
  \end{array}
$$ In particular, if $x$ is actually in $\mani \cap U$, those are
  bases for respectively $E^0 \cap T_x \mani^0$, $E^0$ and $T_x
  \mani^0$. Applying Gram-Schmidt orthogonalization to these families
  of functions, we obtain the functions in the statement of the lemma.
\end{proof}

The following result is a direct corollary of the previous lemma.
\begin{lem}\label{prop:transv}
  Let $\E$ be a Euclidean space, let $\mani\subset\E$ be a manifold of
  class $C^1$ and let $E$ be an affine subspace of
  $\E$. Suppose that $E$ and $\mani$ intersect transversally at a
  point $ \zeta\in E \cap\mani$.  Then, there exists a neighborhood
  $U$ of $\zeta$ such that for $x$ in $U$, $\Pi_\mani(x)$ is
  well-defined and the intersection $E \cap T_{\Pi_\mani(x)} \mani$ is
  not empty.
\end{lem}
\begin{proof}
  Let $U_0$ be a neighborhood of $\zeta$ such that $\Pi_\mani$ is
  well-defined and continuous in $U_0$ and such that the intersection
  $E \cap T_x \mani$ is not empty for $x$ in $\mani \cap U_0$ (such an
  $U_0$ exists by Lemmas~\ref{prop:diffpi}
  and~\ref{lemma:bases}). Then, take $U = \Pi_\mani^{-1}(\mani\cap
  U_0) \cap U_0$.
\end{proof}

In the particular case where $\E=\matsp_{p,q}(\R)$ and
 $\mani=\detvar_r \subset
\matsp_{p,q}(\R)$, the projection $\Pi_{\detvar_r}$ can be made explicit using the
\emph{Eckart-Young Theorem}, which shows that $\Pi_{\detvar_r}(M)$ can
be computed from the \emph{singular value decomposition} of $M$:
\begin{theo}\label{theo:eckart-young}
 Let $M\in\matsp_{p,q}(\R)$ be a matrix, $M=U\cdot S\cdot V^\intercal$
 be its singular value decomposition and
 $\sigma_1\geq\dots\geq\sigma_{\min(p,q)}$ be its singular
 values. Assume that $\sigma_r\neq\sigma_{r+1}$ and let $\widetilde S$
 be the diagonal matrix defined by
 $$\widetilde{S}_{i,i}=\begin{cases}
                        S_{i,i}\text{ if $S_{i,i}\geq \sigma_r$}\\
			 0\text{ otherwise}
                       \end{cases}
$$ Then there exists a unique matrix $\Pi_{\detvar_r}(M)$ of rank $r$
 minimizing the distance to $M$ and this matrix is given by
 $\Pi_{\detvar_r}(M)=U\cdot\widetilde S\cdot V^\intercal$.
\end{theo}

The last notion we will need is the {\em Moore-Penrose pseudoinverse} of
either a matrix or a linear mapping $A$; in both cases, we will denote
it by $A^\dag$. Its main feature is that the solution of a consistent
linear system $A x = y$ with minimal 2-norm is given by $A^\dag y$ (in
the non-consistent case, this outputs the minimizer for the residual
$Ax-y$).

%%%%%%%%%%%%%%%%%%%%%%%%%%%%%%%%%%%%%%%%%%%%%%%%%%%%%%%%%%%%

\subsection{Cadzow's algorithm: alternating projections}

The first occurrence of the general problem of structured low rank
approximation that we are aware of is described in~\cite{Cad88}. In
this paper, Cadzow introduces an algorithm based on alternating
projections to solve SLRA problems. A solution $M'$ of an SLRA problem
should verify two properties:
\begin{itemize}
\item \textbf{(P1)} $M'\in E$;
\item \textbf{(P2)} $\rank(M')\leq r$.
\end{itemize}
Cadzow's algorithm, as illustrated in Figure~\ref{fig:cadzow},
proceeds by looking successively for the nearest matrices which
satisfy alternatively \textbf{(P1)} and \textbf{(P2)}.  The nearest
matrix verifying \textbf{(P1)} is obtained by the orthogonal
projection on $E$, and, as prescribed by the Eckart-Young theorem, the
closest matrix verifying \textbf{(P2)} is obtained by truncating its
Singular Value Decomposition.

We would like to emphasize that in the general case (and in most
applications), the intersection $E \cap \detvar_r$ has positive
dimension, whereas in Figure~\ref{fig:cadzow} (and all further ones),
this intersection appears to have dimension zero.

\begin{figure}[h]
\begin{center}
\begin{tikzpicture}[yscale=1.25,scale=0.4]
%\begin{tikzpicture}[yscale=1.25,scale=0.75]
\begin{axis}[scale=2.3, every axis plot post/.append style={mark=none,ultra thick}, xtick=\empty,ytick=\empty,xmin=-2,xmax=-0.5,ymin=-0.1,ymax=1.4]
\addplot[smooth]{0.4*x^2-0.4*0.5};
\addplot[black]{0};
\addplot[black] coordinates {(-1.9,0) 
(-1.341732934, 0.5100989064)
(-1.341732934, 0.01)
(-1.095900885, 0.2703995000)
(-1.095900885, 0.01)
(-0.9636949230, 0.1614831618)
(-0.9636949230, 0.01)
(-0.8840482001, 0.1026164880)
(-0.8840482001, 0.01)
(-0.8325783947, 0.06727471332)
(-0.8325783947, 0.01)
(-0.7977547626, 0.04456506452)
(-0.7977547626, 0.01)
(-0.7734455500, 0.02928720752)
(-0.7734455500, 0.01)
};
\addplot[black] coordinates {
(-1.391732934, 0.4735189064)
(-1.438312934, 0.523520067660)
(-1.388312934, 0.570100067660)
};
\end{axis}
\node at (4,1.3) {$E$};
\node at (4,9.5) {$\detvar_r$};
\end{tikzpicture}
\caption{Cadzow's algorithm}\label{fig:cadzow}
\end{center}
\end{figure}
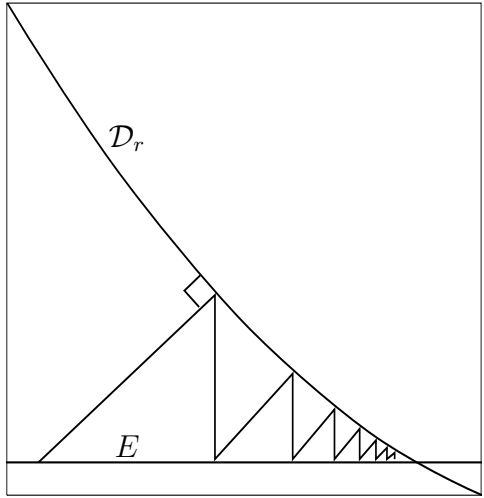

Details of Cazdow's algorithm are given in Algorithm~\ref{algo:cadzow}
below. In this algorithm, for $M\in\matsp_{p,q}(\R)$, the subroutine
$\SVD(M)$ returns three matrices
$U\in\matsp_{p,p}(\R),S\in\matsp_{p,q}(\R),V\in\matsp_{q,q}(\R)$, such
that $M=U\cdot S\cdot V^\intercal$, $U$ and $V$ are unitary matrices,
and $S$ is diagonal. The diagonal entries of $S$ are the singular
values of $M$, sorted in decreasing order.

\begin{algorithm}\caption{one iteration of Cadzow's algorithm}\label{algo:cadzow}
  \begin{algorithmic}[1]
    \Procedure{\tt Cadzow}{$M$, $(E_1,\dots, E_d)$ an orthonormal basis of $E^0$, $r$}
    \State $U,S,V\gets \SVD(M)$
    \State $U_r\gets$ first $r$ columns of $U$
    \State $V_r\gets$ first $r$ columns of $V$
    \State $S_r\gets$ $r\times r$ top-left sub-matrix of $S$
    \State $\widetilde M\gets U_r\cdot S_r\cdot V_r^\intercal$
    \State \textbf{return} $M+\sum_{i=1}^d \langle \widetilde M-M, E_i\rangle E_i$
    \EndProcedure
  \end{algorithmic}
\end{algorithm}

Algorithm \ref{algo:cadzow} (which is sometimes called
\emph{lift-and-project} or \emph{alternating projections} in the
literature) converges linearly towards a matrix $\hat M$ which
verifies both conditions \textbf{(P1)} and \textbf{(P2)}, as proved in
\cite{LewMal08}. In this context, the linear convergence means that if
$(M_i)_{i\geq 0}$ is the sequence of iterates of Cadzow's algorithm
converging towards $\lim_{i\rightarrow\infty}M_i=M_\infty$, then there
exists a positive constant $c$ such that
$$\norm{M_{i+1}-M_\infty}\leq c\norm{M_i-M_\infty}.$$
As pointed out in \cite{ConHir12}, iterating Algorithm \ref{algo:cadzow} does
not converge in general towards $\Pi_{E\cap \detvar_r}(M_0)$. 

%%%%%%%%%%%%%%%%%%%%%%%%%%%%%%%%%%%%%%%%%%%%%%%%%%%%%%%%%%%%

\subsection{Newton's method}

Newton's method is an iterative technique to find zeros of real (or
complex) functions. One of its main features is its quadratic rate of
convergence: each iteration multiplies the number of significant
digits of the solution by two. This iteration was first designed for
systems with as many equations as variables, and was then successfully
extended to non-square systems by using the Moore-Penrose
pseudo-inverse. Let thus $f:\mathbb E\rightarrow \mathbb F$ be a
differentiable mapping of Euclidean spaces. Then the Newton iteration
is given by
$${\sf Newton}_f(x)=x-Df(x)^\dag f(x),$$ where, as said above,
$Df(x)^\dag$ denotes the Moore-Penrose pseudo-inverse of the linear
application $Df(x)$.

In the underdetermined case (when $\dim(\mathbb E)>\dim(\mathbb F)$), this iteration converges locally quadratically towards a point in $f^{-1}({\mathbf 0})$ if $Df(x)$ is locally surjective. The properties of this iteration have been deeply investigated during the last decades \cite{ben1965modified,allgower1990numerical,dedieu2002newton,Ded06}.

Newton's method does not apply directly in our context. However,
Figure~\ref{fig:Newton} below suggests that in some cases (when
for instance $\dim(\detvar_r)=\dim(E)$ and $\detvar_r$ is given as the
graph of a mapping defined on $E$), using Newton iteration could lead
to a fast iterative algorithm. Our algorithm is
motivated by this remark.

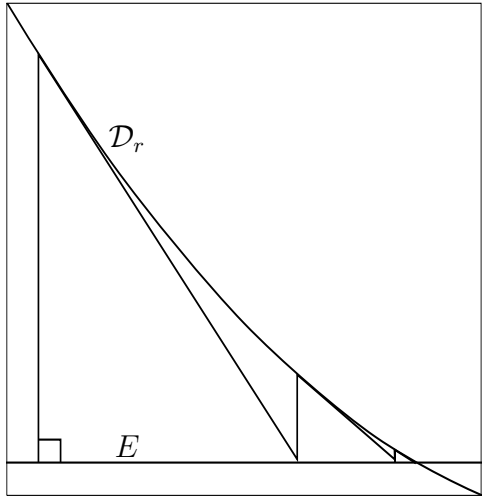
\begin{figure}[H]
\begin{center}
%\begin{tikzpicture}[yscale=1.25,scale=0.75]
\begin{tikzpicture}[yscale=1.25,scale=0.4]
\begin{axis}[
scale=2.3, every axis plot post/.append style=
{mark=none,ultra thick}, xtick=\empty,ytick=\empty,xmin=-2,xmax=-0.5,ymin=-0.1,ymax=1.4]
\addplot[smooth]{0.4*x^2-0.4*0.5};
\addplot[black]{0};
\addplot[black] coordinates {
(-1.9,0)
(-1.9,1.244)
(-1.081578947, 0.01)
(-1.081578947, 0.2679252076)
(-0.7734455500, 0.01)
(-0.7734455500, 0.03928720752)
(-0.7719330257, 0)
(-0.7719330257, 0.03835223848)
(-0.7098288062, 0)
(-0.7098288062, 0.00154277364)
(-0.7071120004, 0)
(-0.7071120004, 0.00000295244)
};

\addplot[black] coordinates {
(-1.9,0.07) (-1.83,0.07) (-1.83, 0)
};
\end{axis}
\node at (4,1.3) {$E$};
\node at (4,9.5) {$\detvar_r$};
\end{tikzpicture}
\caption{Newton's method}\label{fig:Newton}
\end{center}
\end{figure}

%%%%%%%%%%%%%%%%%%%%%%%%%%%%%%%%%%%%%%%%%%%%%%%%%%%%%%%%%%%%
%%%%%%%%%%%%%%%%%%%%%%%%%%%%%%%%%%%%%%%%%%%%%%%%%%%%%%%%%%%%
%%%%%%%%%%%%%%%%%%%%%%%%%%%%%%%%%%%%%%%%%%%%%%%%%%%%%%%%%%%%

\section{Algorithm {\tt NewtonSLRA}}\label{sec:algo}

We propose an iterative algorithm {\tt NewtonSLRA} which combines the
applicability of Cadzow's algorithm and the quadratic convergence of
Newton's iteration. Each of its iterations proceeds in the following three
main steps.
\begin{itemize}
 \item First, compute the projection $\widetilde
   M=\Pi_{\detvar_r}(M)$ onto the determinantal variety $\detvar_r$
   (lines \ref{algSLRA:rankapproxb}-\ref{algSLRA:rankapproxe} in Algorithm~\ref{algo:newtonSLRA});
 \item next, compute a set of generators of the normal space
   $N_{\widetilde M}\detvar_r$ (lines
   \ref{algSLRA:normalb}-\ref{algSLRA:normale});
 \item finally, compute the point in $E \cap T_{\widetilde M}
   \detvar_r$ which minimizes the distance to $M$ (lines
   \ref{algSLRA:projb}-\ref{algSLRA:proje}).
\end{itemize}

We propose two dual methods for computing the last step, leading to
the two variants {\tt NewtonSLRA/1} and {\tt NewtonSLRA/2}
whose pseudo-codes are given in Algorithm~\ref{algo:newtonSLRA} and Algorithm~\ref{algo:newtonSLRA2}. Their main difference is the size of an intermediate matrix
leading to the differences in their domains of efficiency: {\tt NewtonSLRA/1} is
well-suited when $r$ is large and $d$ is small, whereas {\tt
  NewtonSLRA/1} performs better when $r$ is small and $d$ is large.

In Figure~\ref{fig:NewtonSLRA}, we show one iteration of Algorithm {\tt
  NewtonSLRA}; remark that the first step is similar to what happens
in Cadzow's algorithm, but that we then use a linearization inspired
by Newton's iteration. Note as well that in this very particular
example, $E \cap T_{\widetilde M} \detvar_r$ has dimension zero,
whereas this may not be the case in general. Nevertheless, this figure
suggests that our algorithm may converge quadratically (we prove this
rate of convergence in Section~\ref{sec:proofquad}).

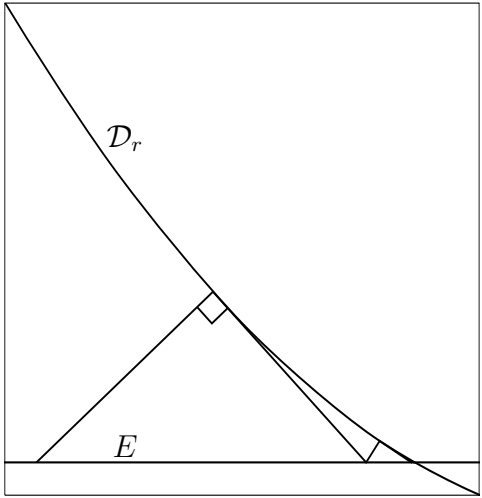
\begin{figure}
\begin{center}
\begin{tikzpicture}[yscale=1.25,scale=0.4]
%\begin{tikzpicture}[yscale=1.25,scale=0.75]
\begin{axis}[
scale=2.3, every axis plot post/.append style=
{mark=none,ultra thick}, xtick=\empty,ytick=\empty,xmin=-2,xmax=-0.5,ymin=-0.1,ymax=1.4]
\addplot[smooth]{0.4*x^2-0.4*0.5};
\addplot[black]{0};
\addplot[black] coordinates {
(-1.9,0)
(-1.341732934, 0.5200989064)
(-0.8571926677, 0)
(-0.8145686697, 0.06540884708)
(-0.7141952305, 0)
(-0.7124620330, 0.00304085940)
(-0.7071269077, 0)
(-0.7071220284, 0.00000862520)
(-0.7071067814, 0)
};

\addplot[black] coordinates {
(-1.391732934, 0.4735189064)
(-1.345152934, 0.423517745140)
(-1.295152934, 0.470097745140)
};
\end{axis}
\node at (4,1.3) {$E$};
\node at (4,9.5) {$\detvar_r$};
\end{tikzpicture}
\caption{{\tt NewtonSLRA}}\label{fig:NewtonSLRA}
\end{center}
\end{figure}

\begin{algorithm}\caption{one iteration of {\tt NewtonSLRA/1} algorithm}\label{algo:newtonSLRA}
      \begin{algorithmic}[1]
    \Procedure{{\tt NewtonSLRA/1}}{$M\in E$, $(E_1,\dots, E_d)$ an orthonormal basis of $E^0$, $r\in\N$}
    \State $(U,S,V)\gets \SVD(M)$ \label{algSLRA:rankapproxb}
    \State $S_r\gets$ $r\times r$ top-left sub-matrix of $S$    
    \State $U_r\gets$ first $r$ columns of $U$
    \State $V_r\gets$ first $r$ columns of $V$
    \State $\widetilde M\gets U_r\cdot S_r\cdot V_r^\intercal$\label{algSLRA:rankapproxe}
    \State $\widetilde{u_1},\ldots,\widetilde{u_{p-r}}\gets$ last $p-r$ columns of $U$\label{algSLRA:normalb}
    \State $\widetilde{v_1},\ldots,\widetilde{v_{q-r}}\gets$ last $q-r$ columns of $V$
    \For{$i\in\{1,\ldots, p-r\}, j\in\{1,\ldots, q-r\}$}
    \State $N_{(i-1)(q-r)+j}\gets \widetilde{u_i}\cdot \widetilde{v_j}^\intercal$\label{algSLRA:gennorm}
    \EndFor\label{algSLRA:normale}
    \State $A\gets (\langle N_k,E_\ell\rangle)_{k,\ell}\in\matsp_{(p-r)(q-r),d}(\R)$\label{algSLRA:projb}
    \State $b\gets (\langle N_k,\widetilde M-M\rangle)_k\in\matsp_{(p-r)(q-r),1}(\R)$
    \State\label{algSLRA:return} \textbf{return} $
    M+
    \sum_{\ell=1}^d \left(A^\dag\cdot b\right)_\ell\,E_\ell$\label{algSLRA:proje}
    \EndProcedure
  \end{algorithmic}
\end{algorithm}

\begin{algorithm}\caption{one iteration of {\tt NewtonSLRA/2} algorithm}\label{algo:newtonSLRA2}
      \begin{algorithmic}[1]
    \Procedure{{\tt NewtonSLRA/2}}{$M\in E$, $(E_1',\dots, E_{pq-d}')$ an orthonormal basis of $(E^0)^\perp$, $r\in\N$}
    \State $(U,S,V)\gets \SVD(M)$ \label{algSLRA2:rankapproxb}
    \State $S_r\gets$ $r\times r$ top-left sub-matrix of $S$    
    \State $U_r\gets$ first $r$ columns of $U$
    \State $V_r\gets$ first $r$ columns of $V$
    \State $\widetilde M\gets U_r\cdot S_r\cdot V_r^\intercal$\label{algSLRA2:rankapproxe}
    \State $u_1,\ldots,u_p\gets$ columns of $U$\label{algSLRA2:normalb}
    \State $v_1,\ldots,v_q\gets$ columns of $V$
    \State $(T_\ell)_{1\leq \ell\leq (p+q-r)r} \gets \text{list of all matrices of the form }u_i\cdot v_j^\intercal$, where $i\leq r$ or $j\leq r$\label{algSLRA2:gennorm}
     \label{algSLRA2:normale}
    \State $A'\gets (\langle E_k',T_\ell\rangle)_{k,\ell}\in\matsp_{pq-d,(p+q-r)r}(\R)$\label{algSLRA2:projb}
    \State $b'\gets (\langle E_k',M-\widetilde M\rangle)_k\in\matsp_{pq-d,1}(\R)$
    \State\label{algSLRA2:return} \textbf{return} $
    \widetilde M+
    \sum_{\ell=1}^{(p+q-r)r} \left(A'^\dag\cdot b'\right)_\ell\,T_\ell$\label{algSLRA2:proje}
    \EndProcedure
  \end{algorithmic}
\end{algorithm}

\smallskip

{\bf Notes on the pseudo-code of {\tt NewtonSLRA}.} While it is convenient to introduce the matrices
$N_{(i-1)(q-r)+j}=\widetilde{u_i}\cdot \widetilde{v_j}^\intercal$ (and $T_\ell$ for the variant {\tt NewtonSLRA/2}) to
prove the correctness of Algorithms~\ref{algo:newtonSLRA} and \ref{algo:newtonSLRA2}, they do not need
to be explicitely computed. All that the algorithm needs are inner
products of the form $\langle N_\ell,X\rangle$ (or $\langle T_\ell, X\rangle$ in {\tt NewtonSLRA/2})
for various matrices $X$. Such an inner product can be computed efficiently by the formula
$$\langle\widetilde{u_i}\cdot
\widetilde{v_j}^\intercal,X\rangle=\widetilde{u_i}^\intercal\cdot
X\cdot \widetilde{v_j}.$$ Also, the Moore-Penrose pseudo-inverses
$A^\dag$ and $A'^\dag$ do not need to be computed: what is actually
needed is the solution of the linear least square problem $\argmin_x
\norm{x}$ subject to $A\cdot x=b$ (resp. $A'\cdot x=b'$). To our
knowledge, using this trick does not change the asymptotic complexity,
but it can make a notable efficiency improvement in practice.

\begin{prop}[Correctness of {\tt NewtonSLRA}]\label{prop:correct}
  Suppose that $\detvar_r$ and $E$ intersect transversally at a point
  $\zeta\in\detvar_r\cap E$.  There exists an open neighborhood $U$ of
  $\zeta$ such that if $M\in U\cap E$ and $(E_1,\dots,E_d)$ is an
  orthonormal basis of $E^0$, then $\Pi_{E \cap T_{\widetilde
      M}\detvar_r}(M)$ is well-defined, for $\widetilde
  M=\Pi_{\detvar_r}(M)$, and Algorithms \ref{algo:newtonSLRA} and \ref{algo:newtonSLRA2} with
  input $(M,(E_1,\cdots,E_d),r)$ return $\Pi_{E \cap T_{\widetilde
      M}\detvar_r}(M)$.
\end{prop}

The proof of this proposition, together with the cost analysis of the
algorithm, occupy the end of this section. Let $U$ be the neighborhood
of $\zeta$ as defined in Lemma~\ref{prop:transv}. In view of that
lemma, $\Pi_{\detvar_r}$ is well-defined on $U$, and so is the mapping
$M\mapsto\Pi_{E \cap T_{\Pi_{\detvar_r}(M)} \detvar_r}(M)$.  In what
follows, we let $\varphi:U\cap E\rightarrow E$ denote the latter
function; thus, our claim is that Algorithm {\tt NewtonSLRA} computes
the mapping $\varphi$.

The following classical result yields an explicit description of the
tangent and normal spaces of determinantal varieties. The notation $\Hom(\R^q,\R^p)$ stands for the set of $\R$-linear maps from $\R^q$ to $\R^p$.

\begin{lem}\label{prop:dettangent}
 Let $M\in\matsp_{p,q}(\R)$ be such that $\rank(M)=r$. Let $\ell$ be
 the linear application
$$
\begin{array}{rrcl}
  \ell:&\R^q&\longrightarrow&\R^p\\
&v&\longmapsto&M\cdot v
\end{array}
$$
Then the tangent space of $\detvar_r$ at $M$ satisfies
$$
\begin{array}{rcl}
T_M\detvar_r^0&=&\im(\ell)\otimes\R^q+\R^p\otimes\Ker(\ell)^\perp\\
&=&\{\ell'\in\Hom(\R^q,\R^p) \mid \ell'(\Ker(\ell))\subset\im(\ell)\}
\end{array}
$$ and the normal space to $\detvar_r$ at $M$ satisfies
$$N_M\detvar_r^0=\Ker(M^\intercal)\otimes\Ker(M).$$
\end{lem}
\begin{proof}
  Classical references for the proof of these claims 
are \cite{ArbCorGriHar84}, \cite[Section 3]{Eis88} and
  \cite[Ch.~6,\S1]{golubitsky1973stable}.
  We recall the proof of the last claim with the notation used in this paper.
  Let $\{a_1,\ldots,
   a_{p-r}\}$ be a basis of $\Ker(M^\intercal)$, and $\{b_1,\ldots,
   b_{q-r}\}$ be a basis of $\Ker(M)$.  Then the set $\{a_i\otimes
   b_j\}_{i,j}$ is a basis of $\Ker(M^\intercal)\otimes\Ker(M)$.  Now
  let $v\in T_M\detvar_r^0$ be a tangent vector. In view of the first
  claim, it can be rewritten as a finite sum $\sum_k c_k\otimes d_k\in
  T_M\detvar_r^0$ where $c_k\in\im(M)$ or
  $d_k\in\Ker(M)^\perp$. Consequently, $\langle a_i\otimes b_j,
  v\rangle=\sum_k \langle a_i,c_k\rangle\langle b_j,d_k\rangle=0$ and
  thus $\Ker(M^\intercal)\otimes\Ker(M)\subset
  N_M\detvar_r^0$. Finally, since
  $\dim(N_M\detvar_r^0)=(p-r)(q-r)=\dim(\Ker(M^\intercal)\otimes\Ker(M))$,
  we obtain $\Ker(M^\intercal)\otimes\Ker(M)=N_M\detvar_r^0$.
\end{proof}

\begin{proof}[Proof of Proposition \ref{prop:correct}] We are now able to prove the correctness of the two variants of {\tt NewtonSLRA}. As in the
algorithm, let us define $\widetilde M= U_r\cdot S_r\cdot
V_r^\intercal$, where $S_r$ is the $r\times r$ top-left sub-matrix of
$S$, and $U_r$ and $V_r$ are made of the first $r$ columns of
respectively $U$ and $V$. Then, by the Eckart-Young Theorem, for $M\in
U$, the matrix $\widetilde M$ is equal to
$\Pi_{\detvar_r}(M)$. Besides, by construction, the vectors
$\widetilde{u_1},\ldots,\widetilde{u_{p-r}}$
(resp. $\widetilde{v_1},\ldots,\widetilde{v_{q-r}}$) are a basis of
$\Ker(\widetilde M^\intercal)$ (resp. $\Ker(\widetilde M)$). Then, the
previous lemma implies that the matrices $N_\ell$ in {\tt NewtonSLRA/1} (resp. $T_\ell$ in {\tt NewtonSLRA/2}) are a basis of the
normal space $N_{\widetilde M}\detvar_r$ (resp. a basis of the tangent space  $T_{\widetilde M}\detvar_r$).

Let $\varphi(M)$ denote $\Pi_{E \cap T_{\widetilde M} \detvar_r}(M)$. In
order to conclude, we have to prove that the matrix computed at 
line~\ref{algSLRA:return} is $\varphi(M)$, that is, that (with the
notation of the algorithms)
$$M+\sum_{\ell=1}^d \left(A^\dag\cdot b\right)_\ell\,E_\ell
= \widetilde M+
    \sum_{\ell=1}^{(p+q-r)r} \left(A'^\dag\cdot b'\right)_\ell\,T_\ell = \Pi_{E \cap T_{\widetilde M} \detvar_r}(M).$$ An element $F$ of
$\matsp_{p,q}(\R)$ belongs to $E \cap T_{\widetilde M} \detvar_r$ if
and only if $F-M$ is in $E^0$ and $F-\widetilde M$ is in
$T_{\widetilde M} \detvar_r^0$. The first condition is equivalent to
the existence of $a_1,\dots,a_d$ such that $F-M=\sum_{i=j}^{d} a_j
E_j$ and the second one holds when
$$\forall i\in
\{1,\ldots,(p-r)(q-r)\}, \begin{array}[t]{rcl}\langle N_i,
  F-\widetilde M\rangle&=&0;
\end{array}$$ 
taking into account the first constraint, the latter ones become, for
all $i\in \{1,\ldots,(p-r)(q-r)\}$,
$$ \begin{array}[t]{rcl} 
\langle  N_i, M-\widetilde M\rangle+\sum_{j=1}^{d}a_j\langle N_i,  E_j\rangle &=& \langle N_i, M-\widetilde M\rangle+\langle N_i, F-M\rangle\\
&=& 0.
\end{array}$$ 
As in the algorithm, set $$A=\begin{bmatrix}\langle N_1,E_1\rangle&\dots&\langle
 N_1,E_{d}\rangle\\ \vdots&\vdots&\vdots\\ \langle
 N_{(p-r)(q-r)},E_1\rangle&\dots&\langle
 N_{(p-r)(q-r)},E_{d}\rangle
\end{bmatrix}\quad\text{and}\quad b=\begin{bmatrix}\langle N_1, \widetilde M-M\rangle\\\vdots\\\langle N_{(p-r)(q-r)}, \widetilde M-M\rangle
\end{bmatrix}.$$
Then, the previous discussion shows that $F$ belongs to $E \cap T_{ \widetilde
  M}  \detvar_r$ if and only if
$$F = M + \sum_{i=j}^{d} a_j E_j,$$ where $a_1,\dots,a_d$ satisfy the
linear system
$$A\cdot \begin{bmatrix} a_1\\\vdots\\a_{d}\end{bmatrix}=b.$$ By
construction, $\varphi(M)$ is the matrix satisfying these constraints
that minimizes $\norm{\varphi(M)-M}$. Since $(E_1,\ldots, E_d)$ is an
orthonormal basis, $\norm{\varphi(M)-M}^2=\sum_{i=1}^{d} a_i^2$ and
hence the least square condition on $\varphi(M)-M$ amounts to finding
the solution $a_1,\dots,a_d$ of the former linear system that
minimizes the 2-norm (we know that this linear system is consistent,
since $E \cap T_{ \widetilde M} \detvar_r$ is not empty).
The least-square solution can be obtained with the
Moore-Penrose pseudo-inverse of $A$, so we finally deduce that
$$\begin{bmatrix}
  a_1\\\vdots\\a_{d}
\end{bmatrix}=A^\dag\cdot b,$$
and hence $\varphi(M)=M+\sum_{i=1}^d \left(A^\dag\cdot b\right)_i\,E_i.$
This proves the correctness of {\tt NewtoNSLRA/1}.

The correctness of {\tt NewtonSLRA/2} is proved similarly, by writing 
$F-\widetilde M=\sum_{\ell=1}^{(p+q-r)r} a'_\ell T_\ell$
for unknown values $a'_\ell\in\R$. The condition $F-M=\widetilde M-M+\sum_{\ell=1}^{(p+q-r)r} a'_\ell T_\ell\in E$ becomes
$$\langle \widetilde M-M, E_i'\rangle + \sum_{\ell=1}^{(p+q-r)r} a'_\ell \langle T_\ell, E_i'\rangle=0$$
and the rest of the proof is similar to the one above.

\end{proof}

%%%%%%%%%%%%%%%%%%%%%%%%%%%%%%%%%%%%%%%%%%%%%%%%%%%%%%%%%%%%

\paragraph{Complexity.}
All subroutines that appear in {\tt NewtonSLRA} are linear algebra
algorithms. In particular, one iteration needs to compute:
\begin{itemize}
\item the Singular Value Decomposition of the $p\times q$ matrix $M$;
\item the matrix $\widetilde{M}$;
\item $O(d (p-r)(q-r))$ inner products between matrices of size $p\times
  q$ (with $d=\dim(E)$), of the form $\langle
  N_{(i-1)(q-r)+j},E_\ell\rangle$ or $\langle
  N_{(i-1)(q-r)+j},\widetilde{M}-M\rangle$, with
  $N_{(i-1)(q-r)+j}=\widetilde{u_i}\cdot \widetilde{v_j}^\intercal$;
\item the Moore-Penrose pseudoinverse of the $(p-r)(q-r)\times d$
  matrix $A$ ({\tt NewtonSLRA/1}) or the $(p+q-r)r\times (pq-d)$
  matrix $A'$ ({\tt NewtonSLRA/2});
\item the output $\varphi(M)= M+\sum_{i=1}^d \left(A^\dag\cdot b\right)_i\,E_i$.
\end{itemize}

As explained in the introduction, we would want to give simple
complexity statements, counting arithmetic operations
$+,-,\times,\div$ over the reals at unit cost, avoiding the discussion
of accuracy inherent to floating-point arithmetic. This is not
possible for the Singular Value Decomposition, so we will simply take
this computation as a black-box. 

Computing $\widetilde{M}$ can be done in $O(pqr)$ arithmetic
operations.  For
$N_\ell=\widetilde{u_i}\cdot\widetilde{u_j}^\intercal$, the inner
products of the form $\langle N_\ell,X\rangle$ can be computed by the
formula
$$\langle\widetilde{u_i}\cdot
\widetilde{v_j}^\intercal,X\rangle=\widetilde{u_i}^\intercal\cdot
X\cdot \widetilde{v_j}$$ in $O(pq)$ arithmetic operations, for a total
of $O(p q d(p-r) (q-r))$ for the construction of the matrix $A$. 
Similarly, the matrix $A'$ in {\tt NewtonSLRA/2} can be constructed within 
$O(pqr(pq-d)(p+q-r))$ operations.
 The Moore-Penrose pseudoinverse of $A$ (resp. $A'$) can then be
computed in $O(d (p-r)^2(q-r)^2)$ arithmetic operations (resp. $O((pq-d)^2(p+q-r)r)$), and deducing
$\varphi(M)$ can be done in $O(dpq)$ operations in {\tt NewtonSLRA/1} (resp. $O(pqr(p+q-r))$ in {\tt NewtonSLRA/2}).

Altogether, up to the SVD computation, all operations can be achieved
within 
\begin{itemize}
\item $O(pqd(p-r)(q-r)+pqr)$ arithmetic operations for {\tt NewtonSLRA/1};
\item $O(pqr(pq-d)(p+q-r))$ arithmetic operations for {\tt NewtonSLRA/2}.
\end{itemize}
In particular, the cost of {\tt NewtonSLRA/1} is at most quadratic in
the size of the input (specifying the basis $E_1,\dots,E_d$ of $E^0$
requires $O(dpq)$ entries).

%%%%%%%%%%%%%%%%%%%%%%%%%%%%%%%%%%%%%%%%%%%%%%%%%%%%%%%%%%%%
%%%%%%%%%%%%%%%%%%%%%%%%%%%%%%%%%%%%%%%%%%%%%%%%%%%%%%%%%%%%
%%%%%%%%%%%%%%%%%%%%%%%%%%%%%%%%%%%%%%%%%%%%%%%%%%%%%%%%%%%%

\section{Rate of convergence}\label{sec:proofquad}
The aim of this section is to prove the local quadratic convergence of
{\tt NewtonSLRA} and to control the distance between its output and
the optimal solution of the SLRA problem. The results given in this
part of the paper are more general than the SLRA context: as in
\cite{LewMal08}, we will perform our analysis for a manifold $\mani$
in a Euclidean space $\E$ of class $C^3$, instead of
$\detvar_r$; as before, we let $E$ be a proper affine subspace of
$\E$.  We assume without loss of generality that $\mani \ne \E$.

Let $\zeta\in E \cap \mani$ be such that the intersection of $E$ and
$\mani$ is transverse at $\zeta$. By Lemmas~\ref{prop:diffpi}
and~\ref{prop:transv}, we know that in a neighborhood $U$ of $\zeta$,
the mapping $x\mapsto \Pi_\mani(x)$ is well-defined and of class $C^2$, and the intersection $E \cap T_{\Pi_\mani(x)} \mani$ is
not empty. As a result, the projection $\varphi: x \mapsto \Pi_{E \cap
  T_{\Pi_\mani(x)} \mani}(x)$ is itself well-defined over $U$. We saw
in the previous section that in the case $\mani=\detvar_r$, algorithm
{\tt NewtonSLRA} precisely computes the mapping $\varphi$.  In the
more general context of this section, we study the iterates
$\varphi^n=\varphi\circ\dots\circ\varphi$ (which will turn out to be
well-defined, up to restricting the domain of $\varphi$).

The transversality assumption implies that, up to restricting $U$, the
intersection $\maniinter=E \cap \mani \cap U$ is a manifold of class
 $C^3$. Up to restricting $U$ further, we can assume (by means
of Lemma~\ref{prop:diffpi}) that the projection operator
$\Pi_{\maniinter}$ is well-defined and of class $C^2$ in $U$.
In the context of Structured Low Rank Approximation, $\maniinter=E
\cap \detvar_r \cap U$, and the projection $\Pi_{\maniinter}$
represents the optimal solution to our approximation problem.

The following theorems are the main results of this section; taken in
the context of SLRA, they finish proving the theorems stated in the
introduction.

The first part of the following theorem ensures the local quadratic
convergence of the iterates of $\varphi$; the second part bounds the
distance between the limit point of the iteration and the optimal
solution $\Pi_\maniinter(x_0)$. Roughly speaking, this shows that
locally the limit of the iteration looks like the orthogonal
projection on $\maniinter$. This will be formalized in Theorem
\ref{theo:diffPhi}.

\begin{theo}\label{theo:rate}
  Let $\zeta$ be in $E \cap \mani$ such that $\Pi_\mani$ is $C^2$
  around $\zeta$ and $\mani$ and $E$ intersect transversally at
  $\zeta$. There exists $\nu,\gamma,\gamma'>0$ such that, for all
  $x_0\in B_\nu(\zeta)$, the sequence $(x_i)$ given by
  $x_{i+1}=\varphi(x_i)$ is well-defined and converges towards a point
  $x_\infty\in\maniinter$, with
  \begin{itemize} 
  \item $\norm{x_{i+1}-x_\infty} \le \gamma\norm{x_i-x_\infty}^2$ for $i \ge 0$;
  \item $\norm{\Pi_\maniinter(x_0)-x_\infty} \le \gamma'\norm{\Pi_\maniinter(x_0)-x_0}^2$. 
  \end{itemize}
\end{theo}

In general, $x_\infty\neq \Pi_\maniinter(x_0)$; in particular, {\tt NewtonSLRA}
will usually not converge to the optimal solution of an SLRA
problem. Nevertheless, the following theorem shows that $\Phi$ is a
good local approximation of the function $\Pi_\maniinter$ around
$\maniinter$.

\begin{theo}\label{theo:diffPhi}
  Let $\zeta$ be in $E \cap \mani$ such that $\Pi_\mani$ is $C^2$
  around $\zeta$ and $\mani$ and $E$ intersect transversally at
  $\zeta$, and let $\Phi: B_\nu(\zeta) \to \E$ denote the limit
  operator $\Phi(x)=x_\infty$, for $x_\infty$ as in
  Theorem~\ref{theo:rate}. Then, $\Phi$ is differentiable at $\zeta$
  and $D\Phi(\zeta)=\Pi_{T_\zeta\maniinter^0}$.
\end{theo}

Note that in the context of SLRA, $\detvar_r$ 
and $E \cap \detvar_r$ are of class $C^\infty$ in the neighborhood of
points $\zeta\in E \cap\detvar_r$ where the intersection is transverse.

%%%%%%%%%%%%%%%%%%%%%%%%%%%%%%%%%%%%%%%%%%%%%%%%%%%%%%%%%%%%

\subsection{Angle between linear subspaces}

Our analysis will rely on the notion of {\em angle} between two linear
subspaces (see \emph{e.g.} \cite{Fri37}, \cite[Ch.~9]{Deutsch01}, \cite[Section
  3]{LewMal08}).  In what follows, $\Sphere=\{x\in\E: \norm{x}=1\}$
denotes the unit sphere and $M^\perp$ denotes the orthogonal
complement of a linear subspace $M$ of $\E$.
\begin{defi}[angle between linear subspaces]
  Let $M, N\subset\E$ be two linear subspaces.  If $N\subset M$ or
  $M\subset N$, we set $\alpha(M,N)=0$.  Otherwise, their \emph{angle}
  $\alpha(M,N)$ is the value in $[0,\pi/2]$ defined by
  $$\alpha(M,N):=\arccos\left(\max\{\langle x,y\rangle:x\in\Sphere\cap
  M\cap (M\cap N)^\perp, y\in\Sphere\cap N\cap (M\cap
  N)^\perp\}\right).$$
\end{defi}
The following lemma (see \cite[Lemma 9.5]{Deutsch01} for a proof) shows that when we consider the maximum of the
scalar products, we only need one vector to be orthogonal to $M\cap
N$.
\begin{lem}\label{lem:angle2}
 If $x$ is in $\Sphere\cap M\cap (M\cap N)^\perp$ and $y$ is in
 $\Sphere\cap N$, then
 $$\langle x,y\rangle\leq \cos(\alpha(M,N)).$$
\end{lem}
 
We can now describe a few consequences of our transversality
assumptions for angles between various subspaces.
\begin{lem}\label{lem:transverse}
  There exists an open neighborhood $U$ of $\zeta$ such that
  $\inf_{x\in \mani \cap U}\alpha(T_x \mani^0,E^0)>0$.
\end{lem}
\begin{proof} 
  First, notice that the angle $\alpha(M,N)$ between two linear
  subspaces $M$ and $N$ cannot be $0$ if $M\not\subset N$ and
  $N\not\subset M$. Since by assumption $\mani\neq \E$ and $E\neq \E$,
  and $\mani$ and $E$ intersect transversely at $\zeta$, we have
  neither $T_\zeta \mani^0 \subset E^0$ nor $E^0 \subset T_\zeta
  \mani^0$.  We deduce that $\alpha(T_\zeta^0 \mani, E^0)\neq 0$.

  The rest of the proof is similar to that of~\cite[Lemma
    10]{LewMal08}.  Recall from Lemma~\ref{lemma:bases} that for $x$
  in a neighborhood $U_0$ of $\zeta$, we know orthonormal families
  $(e_1(x),\dots,e_t(x))$,
  $(e_1(x),\dots,e_t(x),e'_{t+1}(x),\dots,e'_d(x))$ and
  $(e_1(x),\dots,e_t(x),e''_{t+1}(x),\dots,e''_s(x))$, that vary
  continuously with $x$, and that are bases of respectively $E^0 \cap
  T_x \mani^0$, $E^0$ and $T_x \mani^0$ whenever $x$ is in $\mani \cap
  U_0$.

  For $x$ in $U_0$, consider the linear mapping $\pi_x=\Pi_{S'(x)}
  \Pi_{S''(x)} - \Pi_{S(x)}$, where $S(x),S'(x),S''(x)$ are the vector
  spaces spanned by the three families above. The matrix of this
  linear mapping, and thus its operator norm, vary continuously with $x$.

  Now, when $x$ is in $\mani\cap U_0$, $\pi_x$ is the linear mapping
  $\Pi_{E^0} \Pi_{T_x \mani^0} - \Pi_{E^0 \cap T_x \mani^0}$.
  From~\cite[Ch.~9]{Deutsch01}, we know that the norm of this
  operator is the cosine of $\alpha(T_x \mani^0,E^0)$. This shows that
  at $x=\zeta$, the norm of $\pi_x$ is nonzero; by continuity, this 
  remains true in a neighborhood $U\subset U_0$ of $\zeta$.
\end{proof}

\begin{lem}\label{lem:trivial}
  There exists an open neighborhood $U$ of $\zeta$ such that for any $x$ and
  $y$ in $\mani \cap U$, the intersection of the vector spaces $E^0
  \cap T_x\mani^0 $ and $ \left(E^0 \cap T_y\mani^0 \right)^\perp$ is
  trivial.
\end{lem}
\begin{proof}
  Let $n=\dim(\E)$, $d=\dim(E)$, $s=\dim(\mani)$ and $t=\dim(E^0 \cap
  T_\zeta \mani^0)$; the transversality assumption shows that
  $t=s+d-n$.
  
  Using again Lemma~\ref{lemma:bases}, we know that there exist a
  neighborhood $U_0$ of $\zeta$ and vectors $e_1(x),\dots,e_t(x)$
  depending continuously of $x \in U_0$, that form a orthonormal
  family, and whose span is $E^0 \cap T_x\mani^0$ for $x$ in $\mani
  \cap U_0$. Then, up to restricting further $U_0$, we consider a
  local submersion $\psi:\E\rightarrow \R^{n-t}$ such that
  $\psi^{-1}(0)\cap U_0=E\cap \mani\cap U_0$. Applying Gram-Schmidt
  orthogonalisation to the gradient of $\psi$ defines vectors
  $e_{t+1}(x),\dots,e_n(x)$ that depend continuously on $x$ and such
  that $(e_1(x),\dots,e_n(x))$ is an orthonormal basis of $\E$. In
  particular, when $x$ is in $\mani \cap U_0$,
  $(e_{t+1}(x),\dots,e_n(x))$ is an orthonormal basis of $(E^0 \cap
  T_x\mani^0)^\perp$.

  For $x$ and $y$ in $\mani\cap U_0$, the intersection of $E^0 \cap
  T_x\mani^0$ and $\left(E^0 \cap T_y\mani^0\right)^\perp$ is reduced
  to $\{0\}$ whenever the determinant $\Delta$ of the family
  $(e_1(x),\dots,e_t(x),e_{t+1}(y),\dots,e_n(y))$ is nonzero.
  The determinant $\Delta$ is a continuous function $U_0\times U_0 \to \R$, and
  $\Delta(\zeta,\zeta)$ is nonzero, so there exists a neighborhood
  $\Omega \subset U_0 \times U_0$ of $(\zeta,\zeta)$ that does not
  intersect $\Delta^{-1}(0)$. It is then enough to take $U$ such that
  $U \times U \subset \Omega$.
\end{proof}

\begin{lem}\label{lem:chi}
  Consider the mapping
  $$
\begin{array}{rrcl}
    \Lambda:&\mani\times \mani&\rightarrow&[0,1]\\
    &(x,y)&\mapsto& \cos(\alpha(E^0 \cap T_x\mani^0, \left(E^0 \cap T_y\mani^0\right)^\perp)).
  \end{array}
$$ There exists an open neighborhood $U$ of $\zeta$ and a constant
$\lambda$ such that for $x,y$ in $\mani \cap U$, $\Lambda(x,y)$ is
well-defined, and the inequality $\Lambda(x,y) \le \lambda \norm{x-y}$
holds.
\end{lem}
\begin{proof}
  As before, let $n=\dim(\E)$, $d=\dim(E)$, $s=\dim(\mani)$ and
  $t=\dim(E^0 \cap T_\zeta \mani^0)$. Using Lemma~\ref{lemma:bases},
  we know that there exist $C^2$ functions $e_1,\dots,e_t: U\to \E$,
  defined in a neighborhood $U_0$ of $\zeta$, such that for $x$ in
  $\mani \cap U$, $e_1(x),\dots,e_t(x)$ is a orthonormal basis of $E^0
  \cap T_x\mani^0$. As in the previous lemma, this basis can be
  completed to an orthonormal basis $(e_1(x),\dots,e_n(x))$ of $\E$,
  with functions $e_{t+1},\dots,e_n$ that are still $C^2$
  around~$\zeta$.

  Consider the function $\Gamma: U_0\times U_0 \to \R$, such that
  $\Gamma(x,y)$ is the 2-norm of the linear mapping
  $\Pi_{e_1(x),\dots,e_t(x)} \Pi_{e_{t+1}(y),\dots,e_n(y)}$.  Using
  the previous lemma, up to restricting $U_0$, we may also assume that
  for $x$ and $y$ both in $\mani \cap U_0$, the intersection of the
  vector spaces $E^0 \cap T_x\mani^0$ and $ \left(E^0 \cap T_y\mani^0
  \right)^\perp$ is trivial. Using~\cite[Ch.~9]{Deutsch01} as in
  Lemma~\ref{lem:transverse}, this implies in particular that for such
  $x$ and $y$, $\Lambda(x,y)=\Gamma(x,y)$. Thus, we are going to prove
  that an inequality of the form $\Gamma(x,y) \le C \norm{x-y}$ holds
  for $x$ and $y$ in $\mani \cap U$, for suitable $U \subset U_0$ and
  $C$.
  
  Let $U$ be an open ball centered at $\zeta$, such that $\overline U$
  is contained in $U_0$. Because $e_{t+1},\dots,e_n$ are $C^1$, there
  exists a constant $c\ge 0$ such that $\norm{e_{i}(x)-e_{i}(y)} \le c/n
  \norm {x-y}$ holds for all $x,y$ in $U$ and $i$ in
  $\{t+1,\dots,n\}$. 

  The matrix $P_y$ of the orthogonal projection
  $\Pi_{e_{t+1}(y),\dots,e_n(y)}$ can be written as $P_y = R_y
  R_y^\intercal$, where $R_y$ is the matrix with columns
  $e_{t+1}(y),\dots,e_n(y)$.  In particular, $R_y$ can be rewritten as
  $R_y= R_x+\delta_{x,y}$, with $R_x$ being the matrix with columns
  $e_{t+1}(x),\dots,e_n(x)$ and where the operator norm of $\delta_{x,y}$
  is bounded by $c \norm{x-y}$.
  As a result, $P_y$ can be rewritten as 
  $$
  \begin{array}{ccl}
    P_y &=& R_y R_y^\intercal\\
&=& R_x R_x^\intercal + R_x \delta_{x,y}^\intercal + \delta_{x,y} R_x^\intercal+ \delta_{x,y} \delta_{x,y}^\intercal\\
&=&P_x + \Delta_{x,y},
  \end{array}
  $$ with $\Delta_{x,y}= R_x \delta_{x,y}^\intercal + \delta_{x,y}
  R_x^\intercal+ \delta_{x,y} \delta_{x,y}^\intercal$.  By
  construction, the norm of $\delta_{x,y}$ is bounded by $c\norm{x-y}$, and the norm of $R_x$ is equal to
  $1$. 
Consequently, the norm of $\Delta_{x,y}$ is bounded by $\lambda \Vert x-y\rVert$ on $U$, with $\lambda=2 c + c^2\sup_{x,y\in U}\norm{x-y}$ (up to restricting $U$, $\sup_{x,y\in U}\norm{x-y}$ can be made arbitrarily small).

  Let further $S_x$ be the matrix of the orthogonal projection
  $\Pi_{e_1(x),\dots,e_t(x)}$, and remark that $S_x P_x=0$.  In view
  of the above paragraphs, the matrix $Q_{x,y}$ of the linear mapping
  $\Pi_{e_1(x),\dots,e_t(x)} \Pi_{e_{t+1}(y),\dots,e_n(y)}$ can be
  rewritten as
  $$
  \begin{array}{ccl}
    Q_{x,y} &=& S_x P_y \\
            &=&    S_x P_x+ S_x \Delta_{x,y}    \\
  &=&  S_x \Delta_{x,y}.
  \end{array}
  $$ Because the norm of $\Delta_{x,y}$ is bounded by
  $\lambda\norm{x-y}$, and the norm of an orthogonal projection is at
  most $1$, the norm of $Q_{x,y}$ is also bounded by $\lambda\norm{x-y}$.  This implies that $\Gamma(x,y)$, which is the norm
  of $Q_{x,y}$ is a most $\lambda\norm{x-y}$.
\end{proof}

%%%%%%%%%%%%%%%%%%%%%%%%%%%%%%%%%%%%%%%%%%%%%%%%%%%%%%%%%%%%

\subsection{Analysis of one iteration}

In what follows, we work over an open neighborhood $U$ of $\zeta$ that
has the form $U=B_\rho(\zeta)$, for some $\rho > 0$ chosen such that
\begin{itemize}
\item $\varphi$, $\Pi_\mani$ and $\Pi_\maniinter$ are well-defined in
  the {\em closed} ball $\overline{B_\rho(\zeta)}$, with $\Pi_\mani$
  and $\Pi_\maniinter$ of class $C^2$;
\item the inequality $\alpha(T_v^0 \mani,E^0)>0$ (as in Lemma
  \ref{lem:transverse}) and the conclusions of
  Lemmas~\ref{lemma:bases},~\ref{lem:trivial} and~\ref{lem:chi} hold
  in the closed ball $\overline{B_\rho(\zeta)}$;
\item $\maniinter \cap \overline{B_\rho(\zeta)}$ is closed (for the Euclidean topology).
\end{itemize}
Define the following:
\begin{itemize}
\item $\alpha_0=\inf_{v\in\mani\cap \overline{B_\rho(\zeta)}}\alpha(T_v \mani^0,E^0)$, so that $\alpha_0 > 0$;
\item $C_\mani=\sup_{v\in \overline{B_\rho(\zeta)}}\norm{D^2\Pi_\mani(v)}$;
\item $C_\maniinter=\sup_{z\in B_\rho(\zeta)}\norm{D\Pi_\maniinter(z)}$;
\item $\lambda$ is the constant introduced in Lemma~\ref{lem:chi};
\item $K=  \left ( \frac{C_\mani}{\sin(\alpha_0)} + \sqrt{2}\lambda\right )$
\item $K' = C_\maniinter K$
\item $\delta>0$ is such that $C_\mani^2 \delta^2 \le 1/2$ and $2\delta + K \delta^2 \le \rho$
  hold.
\end{itemize}

\begin{prop}\label{prop:convquad}
  For $x$ in $B_\delta(\zeta)$, the following properties hold:
  \begin{itemize}
  \item $\varphi(x)$ is in $B_{\rho}(\zeta)$, so $\Pi_\maniinter(\varphi(x))$ 
    is well-defined;
  \item $\norm{\varphi(x)-\Pi_\maniinter(x)} \le K \norm{x-\Pi_\maniinter(x)}^2$;
  \item $\norm{\Pi_\maniinter(\varphi(x))-\Pi_\maniinter(x)} \le K' \norm{x-\Pi_\maniinter(x)}^2.$
  \end{itemize}
\end{prop}

The rest of this subsection is devoted to the proof of this
proposition. Thus, we fix $x$ in $B_\delta(\zeta)$ in all that
follows; we also use the following shorthand: $y=\Pi_\mani(x)$,
$w=\Pi_\maniinter(x)$ and $z=\Pi_{T_y \mani}(w)$. Another pair of
points $w'$ and $z'$ will be used: $w'$ is the orthogonal projection
of $x$ on the affine space parallel to $E \cap T_y \mani$ containing
$w$, and $z'=\Pi_{T_y\mani}(w')$.

\paragraph{Step 1: Some basic inequalities.}
First, notice that if $x$ is in $B_{\delta}(\zeta)$, then we have
$$\norm{x-w} \le \norm{x-\zeta} < \delta,$$ 
because $\zeta$ is in $\maniinter$ and $w=\Pi_\maniinter(x)$, and 
$$\norm{x-y} \le \norm{x-\zeta} < \delta,$$ because $\zeta$ is in
$\mani$ and $y=\Pi_\mani(x)$. This implies that $w$ and $y$ belong to
$B_{2\delta}(\zeta)$ and thus to $B_\rho(\zeta)$ since
$$\begin{array}{r}
  \norm{w-\zeta}\leq\norm{w-x}+\norm{x-\zeta} <2\delta \le \rho,\\[1mm]
  \norm{y-\zeta}\leq\norm{y-x}+\norm{x-\zeta} <2\delta \le \rho.
\end{array}$$
Note also for further use that since $\norm{x-w} < \delta$ and
$\norm{x-y}< \delta$, we also have $\norm{y-w} < 2\delta$.

\paragraph{Step 2: Proof of inequality $\norm{z-w}< C_\mani\norm{x-w}^2$.}
We continue by doing a Taylor approximation of $\Pi_\mani$ between
$y$ and $w$.  Since $\Pi_\mani(w)=w$ and $\Pi_\mani(y)=y$, and since
all points of the line segment between $y$ and $w$ are in
$B_\rho(\zeta)$, we obtain
\begin{eqnarray*}
  w-y&=& \Pi_\mani(w)-\Pi_\mani(y)\\
  &=&\Pi_{T_y\mani^0}(w-y)+r,
\end{eqnarray*}
with $\norm{r}\leq \dfrac{C_\mani \norm{w-y}^2}{2}$.  Because $y+
\Pi_{T_y\mani^0}(w-y) = \Pi_{T_y \mani}(w) = z$, this implies
\begin{equation}\label{eq:approxquad}
  \norm{z-w}\leq \frac{C_\mani}{2}\norm{y-w}^2.
\end{equation}
Since we saw previously that $\norm{y-w} \le 2 \delta$, we deduce
in particular that 
$$\norm{z-w}\leq C_\mani \delta \norm{y-w} \leq 2 C_\mani \delta^2.$$
Because $x-y$ is orthogonal to $T_y \mani^0$, it is orthogonal to
$y-z$, and similarly for $w-z$; these relations imply that
$\norm{y-z} \le \norm{x-w}$.  On the other hand, since $w-z$ is
orthogonal to $y-z$, we also have by the Pythagorean theorem
\begin{equation*}
  \norm{y-w}^2=\norm{y-z}^2+\norm{z-w}^2,
\end{equation*}
so that
\begin{equation*}
  \norm{y-w}^2  \leq \norm{x-w}^2+\norm{z-w}^2.
\end{equation*}
From this inequality, using 
the upper bound
$\norm{z-w}\leq 2 C_\mani \delta^2$, we obtain
\begin{eqnarray*}
  \norm{z-w}&\leq& \frac{C_\mani}{2}  \norm{x-w}^2 +  \frac{C_\mani}{2} \norm{z-w}^2 \\
  &\leq &  \frac{C_\mani}{2}  \norm{x-w}^2 +   C_\mani^2 \delta^2 \norm{z-w}\\
  & \leq & \frac{C_\mani}{2}  \norm{x-w}^2 +  \frac 12 \norm{z-w},
\end{eqnarray*}
since $\delta$ is such that $ C_\mani^2\delta^2 \le \frac 12.$ We
finally get, as claimed,
\begin{equation}\label{eq:zw}
  \norm{z-w}< C_\mani\norm{x-w}^2.
\end{equation}

\paragraph{Step 3: Proof of inequality $\norm{\varphi(x)-w'} \le {\norm{z'-w'}}/{\sin(\alpha_0)}$.}
To prove this inequality, let us introduce the angle $\vartheta$ between
$w'-\varphi(x)$ and $z'-\varphi(x)$. First, we prove that
$\cos(\vartheta) \le \cos (\alpha_0)$, by an application of
Lemma~\ref{lem:angle2}.
\begin{itemize}
\item $w'-\varphi(x)$ is in $E^0$, because
  $w'-\varphi(x)=(w'-w)+(w-\varphi(x))$ and both summands are in
  $E^0$. By construction of $w'$, $w'-w$ is in $(E \cap T_y \mani)^0$,
  which is in $E^0$, and $w$ and $\varphi(x)$ are in $E$, so
  $w-\varphi(x)$ is indeed in $E^0$ as well.
\item $z'-\varphi(x)$ is in $T_y \mani^0$, because both $z'$ and $\varphi(x)$ 
  are in $T_y \mani$.
\item $z'-\varphi(x)$ is in the orthogonal complement of $(E \cap T_y
  \mani)^0$, because $z'-\varphi(x)=(z'-w')+(w'-x)+(x-\varphi(x))$,
  which are respectively orthogonal to $T_y\mani^0$, $(E \cap
  T_y\mani)^0$ and $(E \cap T_y \mani)^0$. By Lemma~\ref{lemma:bases},
  $E\cap T_y \mani$ is not empty, and thus $ (E
  \cap T_y\mani)^0=E^0 \cap T_y\mani^0$.
\end{itemize}
Thus, we can apply Lemma~\ref{lem:angle2} to deduce $\cos(\vartheta) \le
\cos (\alpha_0)$, as claimed. Alternatively, $1/\sin(\vartheta) \le
1/\sin(\alpha_0)$.

Second, remark that $w'-z'$ is orthogonal to $\varphi(x)-z'$. Indeed,
the latter is in $T_y \mani^0$, and by construction $w'-z'$ is
orthogonal to $T_y\mani^0$. This proves that
$\norm{w'-z'}=\sin(\vartheta) \norm{w'-\varphi(x)}$, and thus
the inequality
\begin{equation}\label{eq:phiw}
  \norm{\varphi(x)-w'} \le \frac{\norm{z'-w'}}{\sin(\alpha_0)}.
\end{equation}

\paragraph{Step 4: Proof of inequality $\norm{\varphi(x)-w} \le \frac{C_\mani}{\sin(\alpha_0)} \norm{x-w}^2 + \norm{w'-w}$.}
In order to establish this inequality, remark that the vectors $z-w$
and $z'-w'$ have the same norm. Indeed, $z-w$ and $z'-w'$ are
orthogonal to $T_y \mani^0$ by construction of $z$ and $z'$, and both
$w-w'$ and $z-z'$ are in $T_y \mani^0$. Using~\eqref{eq:zw} and~\eqref{eq:phiw}, we deduce
$$\norm{\varphi(x)-w'} \le \frac{\norm{z-w}}{\sin(\alpha_0)} \le
\frac{C_\mani}{\sin(\alpha_0)} \norm{x-w}^2.$$ Using the triangle
inequality $\norm{\varphi(x)-w} \le \norm{\varphi(x)-w'} +
\norm{w'-w}$, we finally deduce
\begin{equation}\label{eq:phiw2}
\norm{\varphi(x)-w} \le \frac{C_\mani}{\sin(\alpha_0)} \norm{x-w}^2 + \norm{w'-w}.
\end{equation}

\paragraph{Step 5: Proof of inequality $\norm{w-w'}\le \lambda\norm{y-w}\norm{x-w}$.}
Let $\vartheta'$ be the angle between the vectors $w'-w$ and $x-w$;
then, because $x-w'$ is orthogonal to $w'-w$, we have
$\norm{w-w'}=\cos(\vartheta')\norm{x-w}$. We claim further that the
inequality
$$\cos(\vartheta') \le \cos(\alpha(E^0 \cap T_y\mani^0,
\left( E^0\cap T_w\mani^0\right)^\perp))$$ holds. Indeed, this
follows from applying Lemma~\ref{lem:angle2} to the vectors $w'-w$
and $x-w$; let us briefly verify that its assumptions are satisfied:
\begin{itemize}
\item $w'-w$ is in $(E \cap T_y\mani)^0=E^0\cap T_y\mani^0$, by construction.
\item $x-w$ is orthogonal to $T_w \maniinter^0$, and the transversality
    assumption (Lemma~\ref{lemma:bases}) shows that $T_w \maniinter=E \cap T_w \mani$.
  \item By Lemma~\ref{lem:trivial}, the vector spaces $E^0 \cap
    T_y\mani^0 $ and $(E^0 \cap T_w\mani^0)^\perp$ have a trivial
    intersection.
\end{itemize}
Using Lemma~\ref{lem:chi}, we deduce the inequality $\cos(\vartheta')
\le \lambda \norm{y-w}$, and thus 
\begin{equation}\label{eq:ww'}
  \norm{w-w'}\le \lambda\norm{y-w}\norm{x-w}.    
\end{equation}

\paragraph{Step 6: Proof of inequality $\norm{y-w} \le \sqrt{2} \norm{x-w}$.}
We established the following inequalities at Step~2:
$$ \norm{z-w} \le C_\mani \delta \norm{y-w}$$
and 
$$ \norm{y-w}^2  \leq \norm{x-w}^2+\norm{z-w}^2.$$
Combining these two inequalities gives 
\begin{eqnarray*}
   \norm{y-w}^2  &\leq& \norm{x-w}^2+\norm{z-w}^2\\
   &\le & \norm{x-w}^2+C_\mani^2 \delta^2 \norm{y-w}^2\\
   &\le & \norm{x-w}^2+  \frac 12 \norm{y-w}^2,
\end{eqnarray*}
since $\delta^2 C_\mani^2 \le 1/2$. Thus, we deduce
\begin{equation}\label{eq:yw}
  \norm{y-w} \le \sqrt{2} \norm{x-w}.
\end{equation}

\paragraph{Step 7: Proof of inequality $ \norm{\varphi(x)-w} \le K \norm{x-w}^2.$}
Combining the results of Steps 4, 5 and 6, we obtain the first
inequality claimed in Proposition~\ref{prop:convquad}:
\begin{eqnarray*}
  \norm{\varphi(x)-w} &\le& \frac{C_\mani}{\sin(\alpha_0)} \norm{x-w}^2 + \norm{w'-w} \\
  &\le & \frac{C_\mani}{\sin(\alpha_0)} \norm{x-w}^2 +  \lambda\norm{y-w}\norm{x-w} \\
  &\le &  \frac{C_\mani}{\sin(\alpha_0)} \norm{x-w}^2 +  \sqrt{2}\lambda\norm{x-w}^2 \\
  &\le& K \norm{x-w}^2.
\end{eqnarray*}

\paragraph{Step 8: Proof that $\varphi(x)$ is in $B_\rho(\zeta)$.}
Next, we prove that $\norm{\zeta-\varphi(x)} < \rho$.  Indeed,
recall that we saw in Step~1 that $\norm{\zeta-w} < 2\delta$ and
$\norm{x-w} < \delta$. Using the inequality proved above, we deduce
$$\begin{array}{rcl}
  \norm{\zeta-\varphi(x)}&\leq&\norm{\zeta-w}+\norm{\varphi(x)-w}\\
  &<&2\delta+K\norm{x-w}^2\\
   &<&2\delta+K\delta^2,
\end{array}$$
which is less than $\rho$ by construction of $\delta$. 

\paragraph{Step 9: Proof of inequality $\norm{\Pi_\maniinter(\varphi(x))-w} \le K'\norm{x-w}^2$.}
This is last item required to conclude the proof of
Proposition~\ref{prop:convquad}.  A first order Taylor expansion of
$\Pi_\maniinter$ along the line segment joining $\varphi(x)$ to $w$
(which are both in $B_\rho(\zeta)$) gives
$$\begin{array}{rcl}
  \norm{\Pi_\maniinter(\varphi(x))-\Pi_\maniinter(w)}&\leq&C_\maniinter\norm{\varphi(x)-w}\\ &\leq&
  K  C_\maniinter\norm{x-w}^2 \qquad \text{\emph{c.f.} Step~7}\\ 
  &\leq& K'\norm{x-w}^2.
\end{array}$$
Since $\Pi_\maniinter(w)=w$, the proof is complete.

%%%%%%%%%%%%%%%%%%%%%%%%%%%%%%%%%%%%%%%%%%%%%%%%%%%%%%%%%%%%

\subsection{Convergence of {\tt NewtonSLRA}}

In this subsection, we study the behavior of the sequence defined by
$x_{i+1} = \varphi(x_i)$: we prove that the sequence is well-defined
for $x_0$ close enough to $\zeta$ and that it converges quadratically
to a limit $x_\infty$. In what follows, we write $\kappa = K + K'$ and
we choose $\nu>0$ such that $\kappa \nu <1/2$ and $4\nu<\delta$.

\begin{prop}\label{prop:iterqud} 
  Let $x_0$ be in $B_\nu(\zeta)$. One can define sequences $(x_i)_{i
    \ge 0}$ and $(w_i)_{i \ge 0}$ of elements of $\E$ such that
  $\norm{x_0-w_0} \le \nu$ and, for $i \ge 0$:
  \begin{itemize}
  \item $x_i$ is in $B_{\delta}(\zeta)$;
  \item $w_{i}=\Pi_\maniinter(x_i)$;
  \item $x_{i}=\varphi(x_{i-1})$ if $i \ge 1$;
  \item $\norm{x_{i}-w_{i}}\leq \kappa \norm{x_{i-1}-w_{i-1}}^2$ if $i \ge 1$;
  \item $\norm{w_{i}-w_{i-1}}\leq \kappa \norm{x_{i-1}-w_{i-1}}^2$ if $i \ge 1$.
\end{itemize}
\end{prop}
\begin{proof}
  We do a proof by induction; precisely, we prove that for all $i\ge
  0$, one can construct $x_1,\dots,x_i$ and $w_0,\dots,w_i$ that
  satisfy the five items above. 

  For $i=0$, the inequality $\norm{x_0-\zeta} \le \delta$ follows from
  the fact that $x_0$ is in $B_\nu(\zeta)$ and that $\nu \le \delta$.
  This implies that $w_0=\Pi_\maniinter(x_0)$ is well-defined; these
  are all the facts we need to prove for index $0$.  In what follows,
  we will also use the facts that $\norm{x_0-w_0} \le \nu$, which
  holds since $w_0$ is the closest point to $x_0$ on $\maniinter$, and
  hence $\norm{w_0-\zeta} \le \norm{w_0-x_0} + \norm{x_0-\zeta}\le 2
  \nu$.

  Let us now assume that the claims hold up to index $i$, and prove
  that they still hold at index $i+1$. Thus, $x_1,\dots,x_i$ and
  $w_0,\dots,w_i$ have been defined, and $x_i$ is in
  $B_{\delta}(\zeta)$.

  We set $x_{i+1}=\varphi(x_i)$; this is valid since $x_i$ is in
  $B_\delta(\zeta)$. For the same reason, we can apply
  Proposition~\ref{prop:convquad}; we deduce that $\norm{x_{i+1}-w_i}
  \le K \norm{x_i-w_i}^2,$ that we can define $w_{i+1} =
  \Pi_\maniinter(x_{i+1})$, and that $\norm{w_{i+1}-w_i} \le K'
  \norm{x_i-w_i}^2$ holds. By the triangle inequality
  $\norm{x_{i+1}-w_{i+1}}\leq \norm{x_{i+1}-w_i}+\norm{w_i-w_{i+1}}$,
  we get
  $$\norm{x_{i+1}-w_{i+1}}\leq \kappa \norm{x_i-w_i}^2$$
  and similarly 
  $$\norm{w_{i+1}-w_{i}}\leq K' \norm{x_i-w_i}^2 \le \kappa
  \norm{x_i-w_i}^2.$$ The only thing left to prove is that $x_{i+1}$
  is in $B_\delta(\zeta)$.  To this effect, remark that we have (by an
  easy induction, and using the fact that $\kappa \nu <1/2$) 
  $$\norm{x_j - w_j} \le \kappa^{2^j-1} \nu^{2^j} \le \frac{\nu}{2^{2^j-1}}$$ 
  for $0 \le j \le i+1$ and
  $$\norm{w_{j+1} - w_j} \le \kappa^{2^{j+1}-1} \nu^{2^{j+1}} \le \frac{\nu}{2^{2^{j+1}-1}}$$ for 
  $0 \le j \le i.$ We deduce
  \begin{eqnarray*}
    \norm{x_{i+1}-\zeta}&\leq&\norm{x_{i+1}-w_{i+1}}+\norm{w_{i+1}-w_i}+\dots+\norm{w_1-w_0}+\norm{w_0-\zeta}\\
    &\leq& \frac{\nu}{2^{2^{i+1}-1}}   +\sum_{j=0}^{i} \frac{\nu}{2^{2^{j+1}-1}}     +2\nu\\
    &\leq& 2 \nu\left(1+\sum_{\ell\in\N} \frac 1{2^\ell}\right)\\
    &\leq& 4\nu\\[1mm]
    &<&\delta \quad\text{because $4 \nu < \delta$}.
  \end{eqnarray*}
\end{proof}

\paragraph{Proof of Theorem~\ref{theo:rate}.} First, we prove that the
sequence $(w_i)$ is a Cauchy sequence. Assume that $x_0$ lies in the
ball $B_\nu(\zeta)$. As a consequence of
Proposition~\ref{prop:iterqud}, we deduce by a simple induction (as we
did during the proof of that proposition) that the following holds for
all $i \ge 0$:
\begin{equation}\label{eq:xiwi}
\norm{x_{i} - w_{i}} \le \frac{\nu}{2^{2^i-1}} \quad\text{and}\quad
  \norm{w_{i+1} - w_i} \le \frac{\nu}{2^{2^{i+1}-1}}.
\end{equation}
We deduce in particular
\begin{equation}\label{eq:xiwinu}
  \norm{x_i - w_i} \le \nu,
\end{equation}
and this in turn allows us to prove (by induction on $\ell$) that for 
all $i,\ell$, the following holds:
\begin{equation}\label{eq:xiwiell}
\norm{x_{i+\ell} - w_{i+\ell}} \le \frac{\norm{x_i-w_i}}{2^{2^\ell-1}}.
\end{equation}
As a first consequence, we have, for all $k,\ell\in\N$, with $k \ge \ell$:
\begin{eqnarray*}
\norm{w_k-w_\ell}&\leq&\sum_{i=0}^{k-\ell-1} \norm{w_{\ell+i+1}-w_{\ell+i}}\\
&\leq&\sum_{i=0}^{\infty} \frac{\nu}{2^{2^{\ell+i+1}-1}}  \quad \text{by Eq.~\eqref{eq:xiwi}}\\
& \le & \frac{\nu}{2^{\ell}}.
\end{eqnarray*}
Therefore, the sequence $(w_i)$ is a Cauchy sequence; since
$\lim_i\norm{x_i-w_i} = 0$, both sequences $(x_i)$ and $(w_i)$
converge to a common limit $x_\infty$. This proves the first claim in
Theorem~\ref{theo:rate}. Furthermore, we obtain the following
estimates:
\begin{eqnarray*}
  \norm{x_\infty-w_{i}} & \leq & \sum_{\ell\in\N}\norm{w_{i+\ell+1}-w_{i+\ell}}\\
  &\le & K' \sum_{\ell\in\N} \norm{x_{i+\ell}-w_{i+\ell}}^2 \qquad \text{by Proposition~\ref{prop:convquad}}\\
  &\le& K' \sum_{\ell \in\N} \frac{ \norm{x_i-w_i}^2 } {2^{2^{\ell+1}-2}}  \qquad \text{by Eq.~\eqref{eq:xiwiell}}\\
  &\le& 2K'\norm{x_i-w_i}^2.
\end{eqnarray*}
In particular, for $i=0$, we get the claim of the theorem that
$\norm{x_\infty-\Pi_\maniinter(x_0)} \le \gamma'
\norm{x_0-\Pi_\maniinter(x_0)}^2$, with $\gamma' = 2 K'$.  Besides,
since $\norm{x_i-w_i} \le \nu$ and $2 \kappa \nu < 1$, we also obtain,
for any $i\ge 0$,
\begin{equation}\label{eq:xinftywi}
  \norm{x_\infty-w_i} \le 2K' \nu \norm{x_i-w_i} \le  \norm{x_i-w_i}.
\end{equation}
Finally, we prove that the convergence for the sequence $(x_i)$ is
quadratic.  Note that since $\maniinter \cap \overline{B_\rho(\zeta)}$
is closed, $x_\infty$ is in $\maniinter$ (as claimed in the theorem).
In particular, $\norm{x_i-w_i} \le \norm{x_i-x_\infty}$. We deduce,
for $i \ge 0$,
\begin{eqnarray*}
  \norm{x_{i+1}-x_\infty}&\leq&\norm{x_{i+1}-w_{i+1}}+\norm{w_{i+1}-x_\infty}\\
  &\le & \norm{x_{i+1}-w_{i+1}} + \norm{x_{i+1}-w_{i+1}} \quad \text{by Eq.~\eqref{eq:xinftywi}}\\
  &\le & 2 \kappa \norm{x_{i}-w_{i}}^2  \quad \text{using Proposition~\ref{prop:iterqud}}\\
  &\le & 2 \kappa\norm{x_{i}-x_\infty}^2. 
\end{eqnarray*}
This proves the last missing item from Theorem~\ref{theo:rate}, with
$\gamma = 2 \kappa$.

\paragraph{Proof of Theorem~\ref{theo:diffPhi}.} 
We prove that $\Phi$ is differentiable at
$\zeta$ (which implies as a by-product its continuity around $\zeta$)
and that its derivative is $\Pi_{T_\zeta \maniinter^0}$.  Let
$C'_\maniinter$ denote the operator norm of the second derivative of
$\Pi_\maniinter$ at $\zeta$, which is well defined since $\Pi_\maniinter$ is of
class $C^2$ in $B_\nu(\zeta)$.

Doing a first order expansion of $\Pi_\maniinter$ between $\zeta$ and
a point $x$ in $B_\nu(\zeta)$, and using Theorem \ref{theo:rate} and
the facts that $\Pi_\maniinter(\zeta)=\Phi(\zeta)=\zeta$ and
$\norm{\Phi(x)-\Pi_\maniinter(x)}\le \gamma'\norm{x-\Pi_\maniinter(x)}^2$
proved above, we get
\begin{eqnarray*}
      \norm{\Phi(x)-\Phi(\zeta)-\Pi_{T_\zeta \maniinter^0}(x-\zeta)}
&\le& \norm{\Phi(x)-\Pi_\maniinter(x)}+\norm{\Pi_\maniinter(x)-\Pi_\maniinter(\zeta) -\Pi_{T_\zeta \maniinter^0}(x-\zeta)}\\
&\le& \gamma'\norm{x-\Pi_\maniinter(x)}^2+ \frac{C'_\maniinter}2 \norm{x-\zeta}^2 \\
&\le& (\gamma'+\frac{C'_\maniinter}2)\norm{x-\zeta}^2\text{\quad since $\norm{x-\Pi_\maniinter(x)}\leq  \norm{x-\zeta}$.}  
\end{eqnarray*}
Our claim, and thus  Theorem~\ref{theo:diffPhi}, are proved.

%%%%%%%%%%%%%%%%%%%%%%%%%%%%%%%%%%%%%%%%%%%%%%%%%%%%%%%%%%%%
%%%%%%%%%%%%%%%%%%%%%%%%%%%%%%%%%%%%%%%%%%%%%%%%%%%%%%%%%%%%
%%%%%%%%%%%%%%%%%%%%%%%%%%%%%%%%%%%%%%%%%%%%%%%%%%%%%%%%%%%%

\section{Applications and experimental results}
\label{sec:expe}

Our algorithm {\tt NewtonSLRA} has been implemented in the {\tt Maple}
computer algebra system. In this section, we describe three
applications of Structured Low-Rank Approximation (approximate GCD,
matrix completion and approximate Hankel matrices) and compare our
implementation to previous state-of-the-art. These experiments show
that our all-purpose algorithm often performs as well as, or better
than, existing solutions in a variety of settings.

All experiments have been conducted on a QUAD-core AMD Opteron 8384
2.7GHz.

\subsection{Univariate approximate GCD}
For $i\in\N$, let $\R[x]_i$ denote the vector space of polynomials with real coefficients of degree at most $i$. 
For $m,n,d\in\N$, let $G_{m,n,d}\subset\R[x]^2$ denote the set
$$G_{m,n,d}=\{(f,g)\in \R[x]_m\times\R[x]_n : \deg(\GCD(f,g))=d\}.$$

We consider the Euclidean norm on $\R[x]_m$ and $\R[x]_m\times\R[x]_n$: if $f=\sum_{i=0}^m f_i x^i$ and $g=\sum_{i=0}^n g_i x^i$, then
$$
\begin{array}{rcl}
  \norm{f}&=&\sqrt{\sum_{i=0}^mf_i^2}\\
  \norm{(f,g)}&=&\sqrt{\norm{f}_m^2+\norm{g}_n^2}
\end{array}
$$

\begin{problem}{Approximate GCD}\label{prob:approxGCD}
  Let $(f,g)\in\R[x]_m\times\R[x]_n$, $d\in\N$.
  Find $(f^*,g^*)\in\R[x]_m\times \R[x]_n$ such that
  $\deg\left(\GCD\left(f^*,g^*\right)\right)=d$ and 
$\norm{\left(f-f^*,g-g^*\right)}$ is ``small''.
\end{problem}

Similarly to SLRA, there are several variants of the approximate GCD
problem. In some articles, the goal is to find a pair $(f^*,g^*)$
which minimizes the distance $\norm{(f-f^*,g-g^*)}$ (see \emph{e.g.}
\cite{Ter09} and references therein).  In particular,
\cite{CheYakGalMou09} yields a certified quadratically convergent
algorithm in the particular case $d=1$ (\emph{i.e.} the resultant of
$f^*$ and $g^*$ vanishes).  Sometimes, the goal is to find, if it
exists, a $\varepsilon$-GCD, \emph{i.e.} a pair $(f^*,g^*)$ such that
$\norm{(f-f^*,g-g^*)}<\varepsilon$ and which have common roots for a
given $\varepsilon>0$, see \emph{e.g.} \cite{BinBoi07,CorWatZhi04}. In
some other contexts, the degree of the GCD is not known in advance and
the goal is to maximize the degree $\deg(\GCD(f^*,g^*))$ provided that
$\norm{(f-f^*,g-g^*)}<\varepsilon$ for a given $\varepsilon>0$
\cite{EmiGalLom97}.

First, we recall the definition of the $d$-th Sylvester matrix of two
univariate polynomials, which is rank-deficient if and only if
$\deg(\GCD(f,g))\geq d$.

\begin{defi}[$d$-th Sylvester matrix]
  Let $(f,g)\in\R[x]_m\times\R[x]_n$ be univariate polynomials
  $f=\sum_{i=0}^m f_i x^i$, $g=\sum_{i=0}^n g_i x^i$. The $d$-th
  Sylvester matrix is the $(m+n-d+1)\times (m+n-2 d+2)$ matrix defined
  by
$$
\SubRes_d(f,g)=\begin{array}{@{}l@{}}
\left.\left[\begin{array}{cccccccccc}
  f_m&0&\dots&0&0&g_n&0&\dots&0&0\\
  f_{m-1}&f_m&\ddots&\vdots&\vdots&g_{n-1}&g_n&\ddots&\vdots&\vdots\\
\vdots&\ddots&\ddots&\ddots&\vdots&\vdots&\ddots&\ddots&\vdots&\vdots\\
\vdots&\ddots&\ddots&\ddots&\vdots&\vdots&\ddots&\ddots&\vdots&\vdots\\
\vdots&\ddots&\ddots&\ddots&\vdots&\vdots&\ddots&\ddots&\vdots&\vdots\\
  0&0&\dots&f_0&f_1&0&0&\dots&g_0&g_1\\
  0&0&\dots&0&f_0&0&0&\dots&0&g_0
  \end{array}\right]\right\}n+m-d+1\\
  \quad\hexbrace{4.3cm}{n-d+1}~~~~\hexbrace{3.7cm}{m-d+1}
\end{array}.$$
\end{defi}

It is well-known that $\deg(\GCD(f,g)) = d$ if and only if
$\rank(\SubRes_d(f,g)) = m+n-2 d+1$ (see \emph{e.g.}~\cite[Section
  2]{KalYanZhi07}). Let then $\detvar_{m+n-2 d+1}$ denote the
determinantal variety of the $(m+n-d+1)\times (m+n-2 d+2)$-matrices of
rank $m+n-2 d+1$. The following corollary is a direct consequence of
this equivalence; it shows that Problem \ref{prob:approxGCD} is a
particular case of SLRA.

\begin{coro}
By identifying $\R[x]_m\times\R[x]_n$ with the linear subspace
$\SubRes_d\left(\R[x]_m\times\R[x]_n\right)$, we have
$G_{m,n,d}=\SubRes_d(\R[x]_m\times\R[x]_n)\cap \detvar_{m+n-2 d+1}$.
\end{coro}

% The manifold $G_{m,n,d}$ is not an algebraic set. Its
% closure $\overline{G_{m,n,d}}$ in the Zariski topology is the set of
% pairs $(f,g)$ such that $\deg(f)\leq m$, $\deg(g)\leq n$,
% $\deg(\GCD(f,g))\geq d$. The subvariety of pairs $(f,g)$ such that
% $\deg(\GCD(f,g))>d$ is singular in $\overline{G_{m,n,d}}$.

% \medskip

Since approximate GCD is a particular case of SLRA, we now report
experimental results which describe the behavior of our {\tt Maple}
implementation of {\tt NewtonSLRA/1} in this context (for this application, we are searching for matrices of corank 1, so the first variant of {\tt NewtonSLRA} has a better complexity).  Given
$m,n,d\in\N$ and $\varepsilon>0$, we construct an instance of the
approximate GCD problem as follows:
\begin{itemize}
\item we generate three polynomials $(\tilde f,\tilde g,\tilde
  h)\in\R[x]_{m-d}\times\R[x]_{n-d}\times \R[x]_d$, with all
  coefficients chosen uniformly at random in the interval $[-10,10]$;
\item we set $\mathfrak f=\tilde f \cdot\tilde h/\norm{(\tilde f
  \cdot\tilde h,\tilde g \cdot\tilde h)}$ and $\mathfrak g=\tilde g
  \cdot\tilde h/\norm{(\tilde f \cdot\tilde h,\tilde g \cdot\tilde
    h)}$, so that $\deg(\GCD(\mathfrak f,\mathfrak g))=d$ and
  $\norm{(\mathfrak f,\mathfrak g)}=1$;
\item we construct $(f,g)$ by adding to each of the coefficients of
  $\mathfrak f$ and $\mathfrak g$ a noise sampled from a Gaussian
  distribution $\mathcal N(0,\varepsilon)$ of standard deviation
  $\varepsilon$.
\end{itemize}

In the sequel, we let $(f,g)\in\R[x]_m\times\R[x]_n$ denote the noisy
data constructed as described above and
$(f^*,g^*)\in\R[x]_m\times\R[x]_n$ denote the pair minimizing
$\norm{(f-f^*,g-g^*)}$ subject to $\deg(\GCD(f^*,g^*))=d$.

In Table \ref{table:convquadapproxGCD}, we compare the steps' sizes of
{\tt NewtonSLRA} with those of {\tt GPGCD}, a state-of-the art
algorithm dedicated to the computation of approximate GCDs \cite{Ter09}. The
experimental results give evidence of the practical quadratic
convergence of {\tt NewtonSLRA}, as predicted by Theorem
\ref{theo:rate}. Experimental results for {\tt GPGCD} seem to indicate
linear convergence, but we would like to point out that {\tt GPGCD}
converges towards a solution of the optimization problem (it finds the
nearest pair of polynomials subject to the degree condition on the
GCD) and hence returns a nearer approximation than {\tt NewtonSLRA}.

\begin{table}
  \centering
  \begin{tabular}{|c|c|c|}
\cline{2-3}
\multicolumn{1}{c|}{}&\multicolumn{2}{|c|}{sizes of iteration steps}\\
\hline
iteration&\texttt{NewtonSLRA}&\texttt{GPGCD}\\
\hline
1&{$\mathbf{0.42\,\,\,10^{-3}}$}&$0.20\,\,\,10^{-2}$\\
2&{$\mathbf{0.19\,\,\,10^{-5}}$}&$0.30\,\,\,10^{-3}$\\
3&{$\mathbf{0.11\,\,\,10^{-9}}$}&$0.15\,\,\,10^{-4}$\\
4&{$\mathbf{0.43\,\,\,10^{-18}}$}&$0.68\,\,\,10^{-6}$\\
5&{$\mathbf{0.10\,\,\,10^{-34}}$}&$0.17\,\,\,10^{-8}$\\
\hline
\end{tabular}
  \caption{Quadratic convergence of {\tt NewtonSLRA}. The polynomials are randomly generated with $m=n=25$, $d=10$, and ${\sf Digits}=100$ in {\tt Maple}.}\label{table:convquadapproxGCD}  
\end{table}

Table \ref{table:convNewtonSLRA} shows the experimental behavior of
{\tt NewtonSLRA} on a small example ($n=m=10$, $r=5$) with
high-precision. Here the computation is stopped when the step size
becomes smaller than $10^{-50}$ or after 50 iterations. The
computations were performed with different values of $\varepsilon$
with {\tt Digits}=120 in {\tt Maple} and each entry of the table is on
average over 20 random instances. For $\varepsilon=0.1$ or
$\varepsilon=1$, {\tt GPGCD} did not converge within 50 iterations for
most of the instances while {\tt NewtonSLRA} converges within
approximately 10 iterations. One iteration of {\tt NewtonSLRA} is slightly
slower than one iteration of {\tt GPGCD} because of the cost of
the singular value decomposition. Consequently, the range of problems
where the quadratic convergence of {\tt NewtonSLRA} yields efficiency
improvements are SLRA problems where linearly convergent algorithms
would require a lot of iterations.  

The third column reports the distance between the output of {\tt
  NewtonSLRA} and its input (the noisy pair of polynomials).  Note
that the squared distance between the initial exact data $(\mathfrak
f,\mathfrak g)$ and the noisy data $(f,g)$ follow a $\chi^2$
distribution with $n+m$ degrees of freedom. Therefore, the expected
magnitude of the noise is $\mathbb E(\norm{(\mathfrak f-f,\mathfrak
  g-g)})=\sqrt{\varepsilon^2(n+m)}=\varepsilon\sqrt{n+m}$. In
Table~\ref{table:convNewtonSLRA}, $m=n=10$ and hence the expected
amplitude of the noise is $2\sqrt{5}\varepsilon$. All the entries in the
third column are below this value, which indicates that on
average, the output of {\tt NewtonSLRA} is actually a better
approximation of the noisy data than the initial exact data
$(\mathfrak f,\mathfrak g)$. Consequently, the quality of the solution
returned by {\tt NewtonSLRA} should be sufficient for many
applications even though it does not solve the associated minimization
problem. The last column of Table \ref{table:convNewtonSLRA} indicates
the distance between $(f',g')$, the output of {\tt NewtonSLRA} and the
nearest solution $(f^*,g^*)$. As predicted by Theorem \ref{theo:rate},
the distance to the nearest solution appears to be quadratic in the
magnitude $\varepsilon$ of the noise.

In order to estimate $(f^*,g^*)$, we use the linearly convergent
certified Gauss-Newton iteration in \cite{YakMasCheAur06}, using
  as the starting point of the iteration the pair $(\mathfrak f,
  \mathfrak g)$. Note that using directly the Gauss-Newton approach
  for the approximate GCD problem requires a good starting point: in
  applicative situations, the pair $(\mathfrak f,\mathfrak g)$ is
  unknown and therefore finding such a good pair with a high degree
  gcd is a difficult problem.

\begin{table}
\centering
\begin{tabular}{|c|c|c|c|c|}
\cline{2-3}
\multicolumn{1}{c|}{}&\multicolumn{2}{|c|}{Nb. iterations}&\multicolumn{2}{c}{}\\
\hline
\multicolumn{1}{|c|}{$\varepsilon$}&{\tt NewtonSLRA}&{\tt GPGCD}&$\norm{(f'-f,g'-g)}$&$\norm{(f'-f^*,g'-g^*)}$\\
\hline
$10^{-10}$&$4.0$&$6.0$&$1.86\,\,\,10^{-10}$&$3.12\,\,\,10^{-19}$\\
$10^{-9}$&$4.0$&$6.6$&$1.93\,\,\,10^{-9}$&$2.98\,\,\,10^{-17}$\\
$10^{-8}$&$4.0$&$7.2$&$2.01\,\,\,10^{-8}$&$3.16\,\,\,10^{-15}$\\
$10^{-7}$&$4.9$&$8.7$&$2.06\,\,\,10^{-7}$&$3.25\,\,\,10^{-13}$\\
$10^{-6}$&$5.0$&$10.0$&$1.62\,\,\,10^{-6}$&$5.45\,\,\,10^{-11}$\\
$10^{-5}$&$5.1$&$11.9$&$1.53\,\,\,10^{-5}$&$1.15\,\,\,10^{-9}$\\
$10^{-4}$&$5.6$&$15.4$&$1.82\,\,\,10^{-4}$&$1.99\,\,\,10^{-7}$\\
$10^{-3}$&$6.3$&$24.4$&$1.76\,\,\,10^{-3}$&$1.96\,\,\,10^{-5}$\\
$10^{-2}$&$7.1$&$37.1$&$1.87\,\,\,10^{-2}$&$3.26\,\,\,10^{-3}$\\
$10^{-1}$&$8.7$&$49.2$&$1.43\,\,\,10^{-1}$&$6.94\,\,\,10^{-2}$\\
$10^{0}$&$11.0$&$50$&$2.42\,\,\,10^{-1}$&$1.71\,\,\,10^{-1}$\\
\hline
\end{tabular}
  \caption{Experimental convergence of {\tt NewtonSLRA}. The pair $(f,g)$ is the input polynomials, $(f',g')$ is the output of {\tt NewtonSLRA}, and $(f^*,g^*)$ is the optimal solution (the polynomials minimizing $\norm{f-f^*,g-g^*}$ under the constraint $\deg(\GCD(f^*,g^*))=d$). The polynomials are randomly generated with $m=n=10$, $d=5$, and ${\sf Digits}=120$ in {\tt Maple}. The iteration is stopped when the step size becomes smaller that $10^{-50}$ or after $50$ iterations}\label{table:convNewtonSLRA}  
\end{table}

We would like to point out that {\tt NewtonSLRA} is also able to solve
larger problems: for instance, it can compute approximate GCDs for
$m=n=2000$ and $d=1000$ within a few minutes (for the default
numerical precision of Maple: {\tt Digits=10}). In order to
  demonstrate the efficiency of our approach, we compare in
  Table~\ref{table:comparUVGCD} timings obtained with our implementation of {\tt
    NewtonSLRA} and with the software {\tt uvGCD}
  \cite{zeng2004approximate}. We observe in this table that {\tt NewtonSLRA} runs faster than {\tt uvGCD}; the quality of the output (\emph{i.e.} the value $\norm{(f_{output}-f,g_{output}-g)}$) of {\tt NewtonSLRA} is comparable to that of {\tt uvGCD}.

\begin{table}
\centering
\begin{tabular}{|c|c|c|c|c|}
\cline{2-5}
\multicolumn{1}{c|}{}&\multicolumn{2}{|c|}{time (in $s$)}&\multicolumn{2}{c|}{$\norm{(f_{output}-f,g_{output}-g)}$}\\
\hline
\multicolumn{1}{|c|}{$(m,n,d)$}&{\tt NewtonSLRA}&{\tt uvGCD}&{\tt NewtonSLRA}&{\tt uvGCD}\\
\hline
\hline
(20,20,10)&0.0264&0.1876&0.0003034353&0.0003034150\\
(40,40,20)&0.0552&0.6185&0.0005466551&0.0004171554\\
(60,60,30)&0.1925000&1.670900&0.0005305&0.0005248752\\
(80,80,40)&0.3870000&3.277600&0.0006652485&0.0006573120\\
(100,100,50)&0.4288000&5.221100&0.0008292970&0.0007893376\\
(120,120,60)&0.5922000&8.987600&0.0008901396&0.0007972352\\
(140,140,70)&0.8618000&12.58410&0.0009151193&0.0008635334\\
(160,160,80)&1.040300&16.84990&0.0009183804&0.001195548\\
(180,180,90)&1.510300&24.01900&0.0009902834&0.001812256\\
(200,200,100)&1.601800&29.00110&0.001041346&0.001032610\\
(220,220,110)&1.970000&39.47140&0.002613709&0.001061010\\
(240,240,120)&2.363400&49.85650&0.001227303&0.001083454\\
(260,260,130)&2.771200&61.15920&0.001224906&0.001153301\\
(280,280,140)&3.419700&73.69030&0.003155242&0.004107356\\
(300,300,150)&3.082400&86.92640&0.001296697&0.003963097\\
\hline
\end{tabular}
\caption{Comparison with {\tt uvGCD}. For both of the software, $(f,g)$ is the pair of input polynomials, $(f_{output},g_{output})$ is the output pair.\label{table:comparUVGCD}}
\end{table}

\subsection{Low-rank matrix completion}
\label{sec:matrixCompletion}

Matrix completion is a problem arising in several applications in
Engineering Sciences, and plays an important role in the recent
development of compressed sensing. Knowing some properties of a matrix
(\emph{e.g.} its rank), the goal is to recover it by looking
only at a subset of its entries. We focus here on low-rank matrix
completion which can be modeled by structured low-rank approximation:
let $I$ be a subset of $\{1,\ldots,p\}\times \{1,\ldots, q\}$ and
$A=(a_{i,j})_{a_{i,j}\in \R, (i,j)\in I}$. We consider the affine
space $E\subset
\matsp_{p,q}(\R)$ of all matrices $(M_{i,j})$ such that, for $(i,j)\in
I$, $M_{i,j}=a_{i,j}$.  Low-rank matrix completion is a SLRA problem
since it asks to find a matrix in $E\cap
\detvar_r$.

One particular case of interest for applications is when there is a
unique solution to the matrix completion problem. In that case,
$(p-r)(q-r)>\dim(E)$. Consequently, the transversality
condition required for the analysis performed in
Section \ref{sec:proofquad} does not hold. Therefore, the results in this
section are mainly experimental observations.

\medskip

Efficient techniques have been developped to tackle the matrix
completion problem via a convex relaxation
(see \cite{candes2010matrix,candes2010power,candes2009exact,Rec11} and
references therein). In this section, we report experimental results
which indicate that Algorithm {\tt NewtonSLRA} can be used to solve
families of low-rank matrix completion problems which cannot be solved
by the convex relaxation. Moreover, we give timings which seem to
indicate that the computational complexity of {\tt NewtonSLRA} is of
the same order of magnitude as that of convex optimization techniques.

We follow \cite[Section 7]{candes2009exact} for the generation of instances of the matrix completion problem:
\begin{itemize}
\item for $r\in\{1,\ldots,p\}$, we generate a $p\times p$ matrix $M=L\cdot R$ of rank $r$ by sampling two matrices $L\in\matsp_{p,r}(\R)$ and $R\in\matsp_{r,p}(\R)$ whose entries follow i.i.d. Gaussian distributions $\mathcal N(0,1)$;
\item we uncover $m$ entries at random in the matrix by sampling a subset $I\subset\{1,\ldots,p\}\times\{1,\ldots, q\}$ of cardinality $m$ uniformly at random;
\item the affine space $E$ is the set of matrices $X=(X_{i,j})$ such that $X_{i,j}=M_{i,j}$ if $(i,j)\in I$.
\end{itemize}

Then we run {\tt NewtonSLRA} by setting as the starting point of the
iteration the matrix $N=(N_{i,j})\in E$ defined by
$$\left\{\begin{array}{lr}
N_{i,j}=M_{i,j}&\text{if $(i,j)\in I$}\\
N_{i,j}=0&\text{otherwise}
\end{array}\right.
$$
and we stop iterating {\tt NewtonSLRA} when the size of an iteration
 becomes smaller than $10^{-4}$ or after 100 iterations.  We consider
 the problem solved if {\tt NewtonSLRA} returns a matrix $\hat M$ such
 that
$$\norm{\hat M-M}/\norm{M}<10^{-3}$$
for more than for 75\% of randomly generated instances.

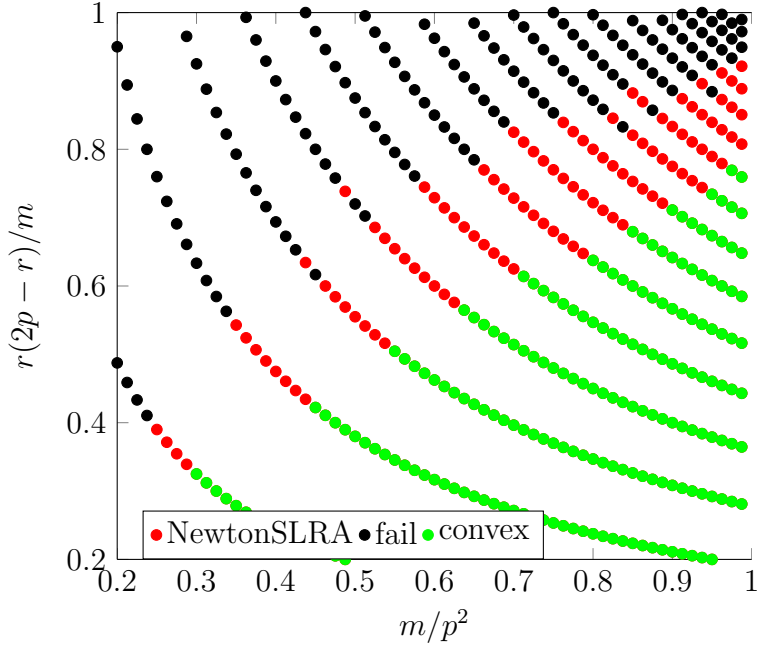
\begin{figure}
\centering
\begin{tikzpicture}[scale=1]
\pgfplotsset{every axis legend/.append style={at={(0.04,0)},anchor=south west}}
\begin{axis}[scale only axis,xmin=0.2,ymin=0.2, ymax=1, xmax=1,axis y line*=left,xlabel={\normalsize $m/p^2$},
legend columns=3, ylabel={\normalsize $r(2p-r)/m$},
]
\addplot[red, mark=*,thin,only marks] table[x=abs,y=ord]{improv};
\addplot[black, mark=*,thin,only marks] table[x=abs,y=ord]{fail};
\addplot[green, mark=*,thin,only marks] table[x=abs,y=ord]{alre};
\legend{NewtonSLRA, fail, convex}
\end{axis}
\end{tikzpicture}\caption{Performances of {\tt NewtonSLRA} for low-rank matrix completion \label{figure:matrixCompletion}}
\end{figure}

Figure \ref{figure:matrixCompletion} reports experimental results for
$n=40$, and should be compared with \cite[Figure
1]{candes2009exact}. Green dots correspond to instances that can be
solved by convex methods (green dots correspond to the white/grey area
in \cite[Figure 1]{candes2009exact}). Any instance that could be
solved by the convex relaxation presented in \cite{candes2009exact} is
also solved by {\tt NewtonSLRA}. Red dots correspond to parameters
where the matrix can be completed by {\tt NewtonSLRA} but not by the
convex relaxation. Black dots correspond to problems which are not
solved by any of these methods. This figure indicates that {\tt
NewtonSLRA} extends the range of matrix completion problems that could
be treated by convex relaxation.

Timings given in \cite{recht2008necessary} indicate that the
semidefinite program obtained via the convex relaxation is solved in
approximately 2 minutes on a 2GHz laptop (for the instances that can
be solved by this method: the green dots in
Figure \ref{figure:matrixCompletion}). For {\tt NewtonSLRA}, the
timings for solving these instances range between 0.8 seconds and 34 seconds
seconds on a QUAD-core Intel i5-3570 3.4GHz.

We also compare our implementation of {\tt NewtonSLRA} with a
  state-of-the-art Matlab software of Riemannian optimization
  developped by B. Vandereycken. Figure \ref{fig:comparVandereycken}
  shows the convergence properties of these two algorithms on an
  example of matrix completion of a $100\times 100$ matrix of rank $5$
  where $1950$ samples have been observed. The graph shows the fast
  convergence of {\tt NewtonSLRA} at each iteration and suggests
  quadratic convergence. The precision is capped at $2^{-48}$ (the
  size of the mantissa of a double float) in order to use BLAS
  routines to compute efficiently the SVD. In practice, the Riemannian
  optimization software is faster than {\tt NewtonSLRA}, even though
  it requires more iterations. For the example described in Figure
  \ref{fig:comparVandereycken}, the total running time of the
  Riemannian optimization software is 0.1s, whereas the total running
  total running time of NewtonSLRA is 45s. Also, {\tt NewtonSLRA} is
  restricted in practice to small matrix sizes and do not apply to
  large-scale matrix completion problems. Consequently, the fast
  convergence of {\tt NewtonSLRA} may yield improvements in
  applications where we require a very precise completion of a small
  size matrix. Also, we would like to point out that {\tt NewtonSLRA}
  also experimentally converges when the input matrix is slightly
  noisy, but this phenomenon is beyond the scope of this paper and is
  not explained by the theoretical convergence analysis in
  Section~\ref{sec:proofquad}.

\begin{figure}
\centering
\begin{tikzpicture}[scale=0.6]
%\pgfplotsset{every axis legend/.append style={at={(0.1,0)},anchor=south west}}
\begin{semilogyaxis}[scale only axis,axis y line*=left,xmin=0,ymax=0,ymin=0.00000000000000001, xlabel={\normalsize iterations of Riemannian optimization},
legend columns=3, ylabel={\normalsize relative residual},
]
\addplot[black, mark=+] table[x=abs,y=ord]{convVanderEycken};
%\legend{NewtonSLRA, fail, convex}
\end{semilogyaxis}
\end{tikzpicture}
\begin{tikzpicture}[scale=0.6]
%\pgfplotsset{every axis legend/.append style={at={(0.1,0)},anchor=south west}}
\begin{semilogyaxis}[scale only axis,axis y line*=left,ymax=0,ymin=0.00000000000000001,xmin=0,xlabel={\normalsize iterations of NewtonSLRA},
legend columns=3, ylabel={\normalsize relative residual},
]
\addplot[black, mark=+] table[x=abs,y=ord]{convNewtonSLRA};
%\legend{NewtonSLRA, fail, convex}
\end{semilogyaxis}
\end{tikzpicture}

\caption{Comparison of Riemannian optimization and NewtonSLRA. The relative residual is the value $\norm{P_\Omega(M_{input}-M_i)}/\norm{P_\Omega(M_{input})}$, where $P_\Omega$ is the orthogonal projection on the linear space of matrices having zeroes outside the set of observed entries \cite{Van13}. During {\tt NewtonSLRA}, the relative residual is measured after the computation of the SVD (matrix $\widetilde M$ in the pseudocode).\label{fig:comparVandereycken}}
\end{figure}
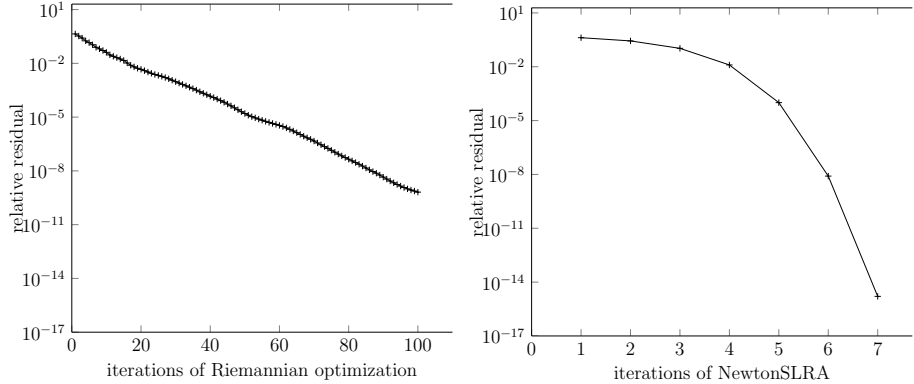

\subsection{Low-rank approximation of Hankel matrices}
In this section, we finally compare the performances of NewtonSLRA
with the STLN approach for Low-Rank Approximation of Hankel matrices
proposed in \cite{park1999low}.  Let us recall briefly the
experimental setting described in \cite[Section 4.2]{park1999low} for
$7\times 5$ Hankel matrices.

Let $H_c$ be the following rank $4$ Hankel matrix:
$$H_c = \begin{bmatrix}
\nu_1&\nu_2&\dots&\nu_5\\
\nu_2&\nu_3&\ldots&\nu_6\\
\vdots&\vdots&\vdots&\vdots\\
\nu_7&\nu_8&\dots&\nu_{11}
\end{bmatrix},$$
where $\nu_i=\sum_{\ell=1}^4 \beta_\ell z_\ell^i$, with $\beta=(1,2,1/2,3/2)$, $z=(\exp(-0.1), \exp(-0.2),\exp(-0.3),\exp(-0.35))$.

The perturbed matrix is $H=H_c+\tau \Delta$, where $\tau>0$ and $\Delta$ is a Hankel matrix with entries picked uniformly at random in the interval $[0,1]$.

In Table \ref{table:Hankel}, we report the number of iterations needed
to obtain a rank $4$ approximation of $H$ with several algorithms. As
in \cite{park1999low}, we stop iterating as soon as the smallest
singular value becomes less than $10^{-14}$. The number of iterations
of NewtonSLRA becomes smaller than for STLN when the magnitude of the
noise becomes larger.

\begin{table}
\centering
\begin{tabular}{|c||c|c|c|c|}
\hline
$\tau$&STLN1&STLN2&Cadzow&NewtonSLRA\\
\hline\hline
$10^{-8}$&1.1&1.7&59.8&{\bf 2.4}\\
\hline
$10^{-7}$&1.6&2.3&75.3&{\bf 3.4}\\
\hline
$10^{-6}$&2.2&2.2&83.0&{\bf 3.9}\\
\hline
$10^{-5}$&2.1&3.2&92.4&{\bf 3.8}\\
\hline
$10^{-4}$&2.1&3.9&93.3&{\bf 4.0}\\
\hline
$10^{-3}$&4.0&6.8&100*&{\bf 4.1}\\
\hline
$10^{-2}$&4.5&20.5&100*&{\bf 4.2}\\
\hline
$10^{-1}$&6.9&22.6&100*&{\bf 4.2}\\
\hline
\end{tabular}
\caption{Number of iterations required by several algorithms to converge towards a rank 4 Hankel matrix\label{table:Hankel}. Each entry in the last column is the average of 30 test results. The three first columns recall the experimental results in \cite[Table 4.1]{park1999low}. 100* means that the algorithm did not converge within 100 iterations.}
\end{table}

Another setting which is important for practical applications is the
behavior of the algorithm in the presence of an outlier, \emph{i.e.}
when one measure is very imprecise compared to the other measures. To
investigate this case, we follow the experimental setting
in \cite[Table 4.2]{park1999low}: we generate Hankel matrices as
above, but then we add 0.01 to all entries on the 8th antidiagonal.
Experiments seem to indicate that {\tt NewtonSLRA} also behaves well
in the presence of such an outlier, as shown by the number of
iterations that we report in Table~\ref{table:Hankeloutlier}. Also,
each of these low-rank approximations of Hankel matrices (with and
without an outlier) were computed in less than 0.6 seconds with {\tt
NewtonSLRA}.

\begin{table}
\centering
\begin{tabular}{|c||c|c|c|c|}
\hline
$\tau$&STLN1&STLN2&Cadzow&NewtonSLRA\\
\hline\hline
$10^{-8}$&8&100*&95&{\bf 4}\\
\hline
$10^{-7}$&8&100*&100*&{\bf 4}\\
\hline
$10^{-6}$&8&100*&90&{\bf 4}\\
\hline
$10^{-5}$&8&100*&95&{\bf 4}\\
\hline
$10^{-4}$&8&100*&99&{\bf 4}\\
\hline
$10^{-3}$&6&100*&100*&{\bf 4}\\
\hline
$10^{-2}$&20&100*&100*&{\bf 4.1}\\
\hline
$10^{-1}$&10&100*&100*&{\bf 4.4}\\
\hline
\end{tabular}
\caption{Number of iterations required by several algorithms to converge towards a rank 4 Hankel matrix \label{table:Hankeloutlier} in presence of an outlier on the 8th antidiagonal. Each entry in the last column is the average of 30 test results. The three first columns recall the experimental results in \cite[Table 4.2]{park1999low}.}
\end{table}

\bigskip

{\bf Acknowledgments.} We are grateful to Erich Kaltofen, Giorgio
Ottaviani, Olivier Ruatta, Bruno Salvy, Bernd Sturmfels and Agnes
Szanto for useful discussions and for pointing out important
references. We also wish to thank an anonymous referee for his useful comments which led to the second variant of the algorithm. We acknowledge the financial support of NSERC and of the
Canada Research Chairs program.

\bibliographystyle{plain}
\bibliography{biblioSLRA}

\bigskip
\bigskip

\footnotesize
\noindent {\bf Authors' addresses:}

\noindent \'Eric Schost, Western University, Department of Computer Science, Middlesex College; Ontario N6A 3K7, Canada. {\tt eschost@uwo.ca}

\smallskip

\noindent Pierre-Jean Spaenlehauer, CARAMEL project, INRIA Grand-Est;
LORIA, Campus scientifique, BP239, 54506 Vandoeuvre-l\`es-Nancy, France, {\tt pierre-jean.spaenlehauer@.inria.fr}
\end{document}